\documentclass[a4paper,10pt]{amsart}

\usepackage{amssymb, enumerate, amssymb, mathrsfs}
\usepackage{amstext}
\usepackage{amsmath}
\usepackage{amscd}
\usepackage{latexsym}
\usepackage{amsfonts}
\usepackage{amsthm}
\usepackage{fancybox}
\usepackage[all]{xy}
\usepackage{lscape}

\newtheorem{thm}{Theorem}[section]
\newtheorem{defn}[thm]{Definition}
\newtheorem{ex}[thm]{Example}
\newtheorem{lem}[thm]{Lemma}
\newtheorem{prop}[thm]{Proposition}
\newtheorem{cor}[thm]{Corollary}
\newtheorem*{que}{Question}

\numberwithin{equation}{section}

\title{String operations on rational Gorenstein spaces}
\author{Takahito Naito}
\date{}
\address{Department of Mathematical Sciences, Shinshu University, 3-1-1 Asahi, Matsumoto, Nagano 390-8621, Japan}
\email{naito@math.shinshu-u.ac.jp}

\begin{document}
\maketitle

\begin{abstract}
F\'{e}lix and Thomas developed string topology of Chas and Sullivan on simply-connected Gorenstein spaces. In this paper, we prove that the degree shifted homology of the free loop space of a simply-connected ${\mathbb Q}$-Gorenstein space with rational coefficient is a non-unital and non-counital Frobenius algebra by solving the up to constant problem. We also investigate triviality or non-triviality of the loop product and coproduct of particular Gorenstein spaces.
\end{abstract}

\section{Introduction}
Chas and Sullivan \cite{CS} introduced a new algebraic structure which is an intersection type product (called the loop product) on the shifted singular homology ${\mathbb H}_{*}(LM)=H_{*+d}(LM) $ of the free loop space $LM={\rm Map}(S^{1},M)$ of any closed oriented $d$-manifold $M$.
Cohen and Godin \cite{CG2004} generalized the product which is called string operations and showed that $H_{*}(LM)$ has the structure of a noncounital commutative Frobenius algebra. It thus gives rise to a 2-dimensional topological quantum field theory without counit. The coproduct of $H_{*}(LM)$ is called the loop coproduct.\\
\indent
After the appearance of \cite{CS}, several authors have extended the theory of string topology.
For instance, Chataur and Menichi \cite{CM2008} considered string topology of classifying spaces. They also proved in the article that the singular cohomology of the free loop space of the classifying space of a compact connected Lie group or a finite group with field coefficients is a homological conformal field theory. In \cite{FT2009}, F\'{e}lix and Thomas developed string topology on Gorenstein spaces. A Gorenstein space was defined by F\'{e}lix, Halperin and Thomas in \cite{FHT1988} and, for instance, closed oriented manifolds, the classifying spaces of connected Lie groups and Borel constructions of Poincar\'{e} duality spaces are Gorenstein spaces. 
The following question is proposed by F\'{e}lix and Thomas; see \cite[p.423]{FT2009}.

\begin{que}
Do the string operations of the homology of the free loop space of a simply-connected Gorenstein space give rise to a 2-dimensional topological quantum field theory?
\end{que}

To answer the question, it is necessary to consider the ``up to constant problem''.
That is to say, there is a problem which involves the strict associative of the loop product and so on. 
The aim of this paper is to give one approach for solving the problem in rational coefficient case.\\
\indent
Let $Dlp$ and $Dlcop$ be the dual loop product and the dual loop coproduct; see \S 2 and \S 3 for the precise definitions. Denote by $(H^{*}(LM;{\mathbb Q}))^{\vee }\cong H_{*}(LM;{\mathbb Q})$ the dual vector space of $H^{*}(LM;{\mathbb Q})$ and by
\begin{align*}
Lp = (Dlp)^{\vee}  : H_{*}(LM;{\mathbb Q})\otimes H_{*}(LM;{\mathbb Q})\longrightarrow H_{*}(LM;{\mathbb Q}),\\
Lcop =(Dlcop)^{\vee}: H_{*}(LM;{\mathbb Q}) \longrightarrow H_{*}(LM;{\mathbb Q})\otimes H_{*}(LM;{\mathbb Q})
\end{align*}
the duals of $Dlp$ and $Dlcop$, respectively.
The following theorem is the main result in this article. Proposition \ref{dlp}, \ref{dlcp} and \ref{Fro} show the following identities (1), (2) and (3).
\begin{thm}\label{mthm} Let $M$ be a simply-connected ${\mathbb Q}$-Gorenstein space of formal dimension $d$. Then,
\begin{enumerate}
\item $Lp (Lp\otimes 1) = (-1)^{d} Lp (1\otimes Lp )$,
\item $(Lcop \otimes 1)Lcop = (-1)^{d}(1\otimes Lcop ) Lcop$,
\item $Lcop \circ Lp = (-1)^{d}(Lp\otimes 1)(1\otimes Lcop) = (-1)^{d}(1\otimes Lp)(Lcop \otimes 1)$.
\end{enumerate}
\end{thm}
Fundamental and important maps which appear in string topology are decomposable with appropriate pull-back diagrams via the Eilenberg-Moore isomorphism; see \cite{FT2009} \cite{KMN}. 
We shall build the proof of our main result Theorem \ref{mthm} by combining the idea with tools in rational homotopy theory.
Actually, the first use of such idea due to Kuribayashi, Menichi and the author \cite{KMN} gives an advantage in the study of string operations.  
\\
\indent
Tamanoi \cite[Theorem A]{Ta2010} proved that if $M$ is an oriented closed smooth manifold, both $(Lcop \otimes 1)Lcop$ and $(1\otimes Lcop ) Lcop$ vanish. However $(Lcop \otimes 1)Lcop$ is not necessarily trivial for Gorenstein spaces. We will give an example of non-trivial case of $(Lcop \otimes 1)Lcop$ in Example \ref{compu}. The equation (3) of Theorem \ref{mthm} is called Frobenius compatibility and compare with Tamanoi's result \cite[Theorem 2.2]{Ta2010}.\\
\indent
We now define the product $m$ on ${\mathbb H}_{*}(LM;{\mathbb Q})=H_{*+d}(LM;{\mathbb Q})$ by
\[
m(a\otimes b) =(-1)^{d(|a|+d)}Lp(a\otimes b).
\]
for $a$, $b \in {\mathbb H}_{*}(LM;{\mathbb Q})$; see \cite[Proposition 4]{CJY2004} and \cite[\S 2]{Ta2010} for the sign. 
By Theorem \ref{mthm} (1) and \cite[Lemma 11.8]{KMN}, we have the following.
\begin{cor}
The shifted homology ${\mathbb H}_{*}(LM;{\mathbb Q})=H_{*+d}(LM;{\mathbb Q})$ endowed with $m$ is an associative graded algebra.
\end{cor}
Another purpose of this paper is to investigate of triviality or non-triviality of the loop (co)product of a simply-connected Gorenstein space.
In \cite{Ta2010}, Tamanoi proved that the loop coproduct of a connected closed oriented manifold is trivial if its Euler characteristic is zero.
F\'{e}lix and Thomas \cite{FT2009} showed that, in rational coefficient, the loop product of the classifying space of a compact connected Lie group is trivial and the dual loop coproduct is non-trivial.
For rational Gorenstein spaces, the torsion functor description of \cite{KMN} enables us to obtain the following two results about triviality or nontriviality of the operations.
\begin{prop}\label{trivial1}
Let $M$ be a simply-connected ${\mathbb Q}$-Gorenstein space with a minimal Sullivan model $(\Lambda V,d)$. Suppose that the rational homotopy group of $M$ is finite dimensional, that is, $V$ is finite dimensional.
\begin{enumerate}
\item If $V$ is generated by odd degree elements, then the loop product is non-trivial and the loop coproduct is trivial.
\item If $V$ is generated by even degree elements, then the loop product is trivial and the loop coproduct is non-trivial.
\end{enumerate}
\end{prop}
We now remark that if $V$ is generated by odd degree elements, $M$ is a Poincar\'{e} duality space since $H^{*}(M;{\mathbb Q})$ is finite dimensional. Therefore, by \cite{KMN}, ${\mathbb H}_{*}(LM;{\mathbb Q})$ is unital and hence the loop product is not trivial. 
\begin{prop}\label{trivial2}
Let $M$ be a simply-connected ${\mathbb Q}$-Gorenstein space whose minimal Sullivan model $(\Lambda V,d)$ is pure. Moreover, assume that both $V^{\text{odd}}$ and $V^{\text{even}}$ are not zero.
\begin{enumerate}
\item If the differential $d$ is zero, then both the loop product and the loop coproduct on $H_{*}(LM;{\mathbb Q})$ are trivial.
\item If ${\rm dim}\, V^{\text{odd}} > {\rm dim}\, V^{\text{even}}$, then the loop coproduct is trivial.
\end{enumerate}
\end{prop}
See Definition \ref{pure} for pure Sullivan models.
For example, even dimensional spheres, complex projective spaces and homogeneous spaces have pure minimal Sullivan models.
As a computational example, we consider the loop (co)product on the Borel construction $ES^{1} \times_{S^{1}}{\mathbb C}P^2$. 
That gives the first example for which both of the product and the coproduct are non trivial in higher degree; 
see Example \ref{compu} and compare with the result \cite[Theorem B (2)]{Ta2010} due to Tamanoi.\\
\indent
The organization of this paper is as follows. In Section 2, we recall a fundamental definitions and facts for string topology of Gorenstein spaces. In Sections 3 and 4, we give the definition of dual loop product and coproduct. A proof of an associativity of the operations is also stated in the section. In Section 5, we show that the loop product and the loop coproduct satisfy Frobenius compatibility. A proof of Proposition \ref{trivial1}, \ref{trivial2} and a computational example are presented in Section 6.


\section{Preliminaries}
We first fix notation and terminology and recall some results for string topology on Gorenstein spaces.
For a topological space $X$, let $LX$ be the free loop space of $X$, let $X^{I}$ be the space consisting of all continuous maps from the closed unit interval $I=[0,1]$ to the space $X$ and let $ev _{0}:LX\to X$ be the evaluation map at $0$. Denote by $LX\times _{X} LX$ the pull-back 
\[
\xymatrix{
LX\times _{X}LX \ar[r]^-{pr_{1}} \ar[d]_-{pr_{2}} & LX \ar[d]^{ev_{0}}\\
LX \ar[r]_{ev_{0}} & X.
}
\]
For a fixed $s\in [0,1]$, let $ev_{i,s}:LX\times _{X}LX \to X$ be the map which is defined as
$ev_{i,s}=ev_{s}\circ pr_{i}$ 
for $i=1,2$. In particular, we put $ev_{0} = ev _{i,0}: LX\times _{X}LX \to X$. \\
\indent
We next recall the definition of Gorenstein spaces. Let $A$ be a differential graded algebra, $M$ and $N$ differential graded $A$-modules. We denote by ${\rm Ext}_{A}(M,N)$ the differential Ext in the sense of Moore, that is, the homology of ${\rm Hom}_{A}(P,M)$, where $P$ is a $A$-semifree resolution of $M$; see \cite[Appendix]{FHT1988}.

\begin{defn}[{\cite[\S 3]{FHT1988}}]
{\rm
A differential graded augmented algebra over a field ${\mathbb K}$, $(A,d)$, is called a {\it Gorenstein algebra} of formal dimension $d$ if
\begin{eqnarray*}
{\rm dim} \, {\rm Ext}_{A}^{*}({\mathbb K}, A) = \left\{ \begin{array}{ll}
 0 & (*\neq d) \\
 1 & (*=d).
 \end{array} \right.
 \end{eqnarray*}
A path-connected space $M$ is called a ${\mathbb K}$-{\it Gorenstein space} of formal dimension $d$ if the normalized singular cochain algebra $C^{*}(M)$ with coefficients in ${\mathbb K}$ is a Gorenstein algebra of formal dimension $d$.
}
\end{defn}

For example, any ${\mathbb K}$-Poincar\'{e} duality space is a ${\mathbb K}$-Gorenstein space. Given $F\to E \to B$ a fibration of simply connected spaces of finite type over a field ${\mathbb K}$. Then, $B$ is a ${\mathbb K}$-Gorenstein space and $F$ is a ${\mathbb K}$-Poincar\'{e} duality space if and only if $E$ is a ${\mathbb K}$-Gorenstein space (\cite[Theorem 4.3]{FHT1988}, \cite[Theorem 1.2]{Mu2007}). In particular, the classifying space of a compact connected Lie group $G$ is a Gorenstein space with formal dimension $-{\rm dim}\, G$.\\
\indent
The following is a key theorem for string topology on Gorenstein spaces.

\begin{thm}[{\cite[Theorem 12]{FT2009}}] \label{key}
Let $M$ be a simply-connected ${\mathbb K}$-Gorenstein space of formal dimension $d$ whose cohomology with coefficients in ${\mathbb K}$ is of finite type. Then
\[
{\rm Ext}_{C^{*}(M^{\times n})}^{*}(C^{*}(M),C^{*}(M^{\times n}))\cong H^{*-(n-1)d}(M)
\]
where $C^{*}(M)$ is considered a $C^{*}(M^{\times n})$-module via the diagonal map $\Delta : M \to M^{\times n}$.
\end{thm}

In this paper, we consider the case that a ground field ${\mathbb K}$ is the field of rational numbers ${\mathbb Q}$.
Now, recall the fundamental facts on rational homotopy theory. For details about rational homotopy theory, see \cite{Rat} or \cite{Rat2} for example.
A {\it minimal Sullivan model} for a simply-connected space $X$ with finite type is a free commutative differential graded algebra over ${\mathbb Q}$, $(\Lambda V ,d)$, with a graded vector space $V=\bigoplus _{i\geq 2}V^{i}$ over ${\mathbb Q}$ where each $V^{i}$ is of finite dimension and $d$ is decomposable; that is, $d(V)\subset \Lambda ^{\geq 2}V$.
Moreover, $(\Lambda V,d)$ is equipped with a quasi-isomorphism $(\Lambda V,d) \stackrel{\simeq }{\longrightarrow } A_{{\rm PL}}(X)$ to the commutative differential graded algebra $A_{{\rm PL}}(X)$ of differential polynomial forms on $X$. Observe that, as algebras, $H^{*}(\Lambda V,d)\cong H^{*}(A_{{\rm PL}}(X))\cong H^{*}(X;{\mathbb Q})$.\\
\indent
 Let $M$ be a simply-connected ${\mathbb Q}$-Gorenstein space of dimension $d$ and denote by $\rho : A=\Lambda V \to A_{{\rm PL}}(M)$ a minimal Sullivan model for $M$. 
Then, by Theorem \ref{key},
\[
H^{d}({\rm Hom}_{A^{\otimes 2}}({\mathbb B}, A^{\otimes 2})) \cong  {\rm Ext}^{d}_{A^{\otimes 2}}(A, A^{\otimes 2}) \cong H^{0}(M)= {\mathbb K},
\]
where $\varepsilon : {\mathbb B}=B(A,A,A)\to A$ is the two-sided bar resolution of $A$ (\cite[Definition 5.51]{Rat2}), and denote by
\[
\Delta ^{!}:{\mathbb B} \longrightarrow A^{\otimes 2}
\]
the map of right $A^{\otimes 2}$-modules which corresponds to a generator of $H^{0}(M)= {\mathbb K}$. We will define the dual loop product and coproduct in \S 2 and \S 3 by using $\Delta ^{!}$.\\
\indent 
We conclude this section by recalling the map which is called the Eilenberg-Moore map in rational coefficient case Consider a pull-back diagram
\[
\xymatrixrowsep{0.5cm}
\xymatrix{
E_{f} \ar[r]^{\tilde{f}} \ar[d]_{q} & E \ar[d]^{p} \\
X \ar[r]^{f} & B
}
\]
in which $p:E\to B$ is a fibration and $B$ is a simply-connected space. The induced map $f^{*}: A_{{\rm PL}}(B)\to  A_{{\rm PL}}(X)$ gives $A_{{\rm PL}}(X)$ a right $A_{{\rm PL}}(B)$-module structure and let $\tilde{\varepsilon } : P\to A_{{\rm PL}}(X) $ be a semifree resolution as left $A_{{\rm PL}}(B)$-modules. Then, we have the quasi-isomorphism called the Eilenberg-Moore map
\[
{\rm EM}:P \otimes _{A_{{\rm PL}}(B)} A_{{\rm PL}}(E) \longrightarrow A_{{\rm PL}}(E_{f})
\]
defined by ${\rm EM}(u\otimes x)=q^{*}\tilde{\varepsilon }(u)\cdot \tilde{f}^{*}(x)$ for $u\otimes x \in P \otimes _{A_{{\rm PL}}(B)} A_{{\rm PL}}(E)$. For details, see \cite{ss}, \cite{Sm1967} for example.


\section{Dual loop product and associativity}
In this section, we first recall the definition of the dual loop product on a simply-connected Gorenstein space $M$. Denote by $\rho :A \to A_{{\rm PL}}(M)$ a minimal Sullivan model for $M$.
The pull-back diagram
\[
\xymatrix{
LM \times _{M} LM \ar[rr]^{{\rm inc}} \ar[d]_{ev_{0}} & & LM \times LM \ar[d]^{ev_{0}\times ev_{0}} \\
M      \ar[rr]^{\Delta }               & & M \times M
}
\]
gives rise to the Eilenberg-Moore map ${\rm EM} : {\mathbb B}\otimes _{A^{\otimes 2}}A_{{\rm PL}}(LM\times LM) \to A_{{\rm PL}}(LM\times _{M}LM)$.
The {\it dual loop product} $Dlp$ is the composite map
\[
\xymatrixrowsep{0.6cm}
\xymatrix{
Dlp : H^{*}(LM) \ar[r]^-{H({\rm Comp} )} & H^{*}(LM\times _{M}LM) \ar[d]^-{{\rm EM}^{-1}}_-{\cong }& \\
 & H^{*}({\mathbb B} \otimes _{A^{\otimes 2}}A_{{\rm PL}}(LM\times LM)) \ar[r]^-{H(\Delta ^{!}\otimes 1)} & H^{*}(LM\times LM)
}
\]
Here $\Delta ^{!}$ is the map stated in \S 2 and the map ${\rm Comp} :LM\times _{M}LM \to LM $ is the concatenation of loops, namely,
\begin{eqnarray*}
{\rm Comp} (\gamma_{1}, \gamma _{2} )(t) = \left\{ \begin{array}{ll}
\gamma _{1} (2t) & (0\leq t\leq \frac{1}{2}) \\
\gamma _{2} (2t-1) & (\frac{1}{2}\leq t\leq 1)
\end{array} \right.
\end{eqnarray*}
for $(\gamma_{1}, \gamma _{2} ) \in LM\times _{M} LM$ and $t\in [0,1]$. 
\begin{prop}\label{dlp}
The dual loop product satisfy the identity
\[
(Dlp \otimes 1)Dlp = (-1)^{d}(1\otimes Dlp)Dlp.
\]
\end{prop}

\begin{proof}
We first consider the commutative diagram:
\[
\xymatrixrowsep{0.5cm}
\xymatrixcolsep{0.2cm}
\xymatrix{
& LM \times _{M} LM \times _{M} LM \ar[rrrr]^{{\rm inc}} \ar[ld]_-{ev_{0}} \ar[dd]_(0.3){=} &&&& LM^{\times 3} \ar[ld]_-{ev_{0}^{\times 3}} \ar[dd]_{=}\\
M \ar[rrrr]^{(\Delta \times 1)\Delta} \ar[dd]_-{=} &&&& M^{\times 3} \ar[dd]_(0.3){=}&\\
& LM \times _{M} LM \times _{M} LM \ar[rr]^-{{\rm inc}} \ar[ld]_{ev_{0}} & & (LM \times _{M}LM)\times LM \ar[rr]^(0.6){{\rm inc}} \ar[ld]_{ev_{0}\times ev_{0}} & & LM^{\times 3} \ar[ld]_{ev_{0}^{\times 3}}\\
M \ar[rr]^{\Delta} & & M^{\times 2} \ar[rr]^-{\Delta \times 1} & & M^{\times 3} & 
}
\]
It is readily seen that the two bottom squares and the top square are pull-back diagrams and we denote by
\begin{align*}
& {\rm EM}_{1}: ({\mathbb B}\otimes A)\otimes _{A^{\otimes 3}}A_{{\rm PL}}(LM ^{\times 3}) \longrightarrow A_{{\rm PL}}((LM \times _{M} LM)\times LM), \\
& {\rm EM}_{2}:{\mathbb B} \otimes _{A^{\otimes 2}}A_{{\rm PL}}((LM \times _{M} LM)\times LM ) \longrightarrow  A_{{\rm PL}}(LM\times _{M} LM \times _{M} LM),\\
& {\rm EM}_{3}: ({\mathbb B}\otimes _{A}{\mathbb B}) \otimes _{A^{\otimes 3}}A_{{\rm PL}}(LM ^{\times 3}) \longrightarrow A_{{\rm PL}}(LM\times _{M} LM \times _{M} LM)
\end{align*}
the Eilenberg-Moore maps of the diagram, respectively. We here remark that $\varepsilon \otimes 1 : {\mathbb B}\otimes A \to A\otimes A$ is a semifree resolution of $A^{\otimes 2}$ as $A^{\otimes 3}$-modules with the $A^{\otimes 3}$-module structure of $A^{\otimes 2}$ given by $(\text{product})\otimes 1:A^{\otimes 3} \to A^{\otimes 2}$, and $\varepsilon \cdot \varepsilon :{\mathbb B}\otimes _{A}{\mathbb B} \to A$ is a semifree resolution of $A$ as $A^{\otimes 3}$-modules.
Then we have the following commutative diagram:
\begin{equation} \label{dlp}
\begin{split}
\xymatrixrowsep{0.7cm}
\xymatrixcolsep{1.2cm}
\xymatrix{
H^{*}(LM) \ar[d]_{H({\rm Comp})} \ar[r]^{=} \ar@{}[rd]|{{\rm (I)}} & H^{*}(LM) \ar[d]^{H({\rm Comp}')}
\\
H^{*}(LM\times _{M}LM) \ar[r]^-{H({\rm Comp}\times 1)}  \ar[d]_{\rm{EM}^{-1}}^{\cong }  \ar@{}[rd]|{{\rm (I\hspace{-.1em}I)}} & H^{*}(LM \times _{M} LM \times _{M} LM)  \ar[d]^{\rm{EM}_{2}^{-1}}_{\cong }\\
H^{*}({\mathbb B}\otimes _{A^{\otimes 2}}A_{{\rm PL}}(LM^{\times 2}))\ar[d]_{H(\Delta^{!} \otimes 1)} \ar[r]^-{ H(1\otimes ({\rm Comp} \times 1)^{*} )} \ar@{}[rdd]|(0.3){{\rm (I\hspace{-.1em}I\hspace{-.1em}I)}}& H^{*}({\mathbb B} \otimes _{A^{\otimes 2}}A_{{\rm PL}}((LM\times _{M} LM) \times LM))  \ar[ldd]_{H(\Delta^{!} \otimes 1)}  \ar[dd]^-{(1\otimes \rm{EM}_{1})^{-1}}_-{\cong } \ar@{}[lddd]|{{\rm (I\hspace{-.1em}V)}}\\
H^{*}(LM^{\times 2}) \ar[d]_{H({\rm Comp} \times 1)} & \\
H^{*}((LM\times _{M}LM)\times LM) \ar[d]_{\rm{EM}_{1}^{-1}}^-{\cong } & H^{*}({\mathbb B} \otimes _{A^{\otimes 2}}(({\mathbb B}\otimes A) \otimes _{A^{\otimes 3}}A_{{\rm PL}}(LM^{\times 3})))  \ar[dl]_{H(\Delta ^{!} \otimes ( 1\otimes 1) \otimes 1)} \ar[dd]^{\Gamma }_{\cong }\\
H^{*}(({\mathbb B}\otimes A)\otimes _{A^{\otimes 3}}A_{{\rm PL}}(LM^{\times 3}))  \ar[d]_{H( ( \Delta ^{!}\otimes 1) \otimes 1)} \ar@{}[rd]|(0.4){{\rm (V)}}& \\
H^{*}(LM^{\times 3})   & H^{*}(({\mathbb B}\otimes _{A} {\mathbb B})\otimes _{A^{\otimes 3}}A_{{\rm PL}}(LM^{\times 3})). \ar[l]_-{H( ( \Delta^{!}\otimes \Delta^{!} ) \otimes 1)} 
}
\end{split}
\end{equation}
Here, the map ${\rm Comp}' : LM\times _{M}LM\times _{M} LM \to LM$ is the composition map given by
\begin{eqnarray*}
{\rm Comp} '  (\gamma_{1}, \gamma _{2} , \gamma _{3})(t) = \left\{ \begin{array}{ll}
 \gamma _{1} (3t) & (0\leq t\leq \frac{1}{3}) \\
 \gamma _{2} (3t-1) & (\frac{1}{3}\leq t\leq \frac{2}{3}) \\
 \gamma _{3} (3t-2) & (\frac{2}{3}\leq t\leq 1).
\end{array} \right.
\end{eqnarray*}
The left $A^{\otimes 2}$-module structure of $({\mathbb B}\otimes A)\otimes _{A^{\otimes 3}} A_{{\rm PL}}(LM^{\times 3})$ given by the natural right $A^{\otimes 2}$-module structure of ${\mathbb B}\otimes A$, that is,
\[
 (b_{1}\otimes b_{2})\cdot ((u\otimes a)\otimes x) = (-1)^{|b_{1}||u|+|b_{2}|(|u|+|a|)}(ub_{1}\otimes ab_{2})\otimes x
\]
for $ (u\otimes a)\otimes x \in ({\mathbb B}\otimes A)\otimes _{A^{\otimes 3}}A_{{\rm PL}}(LM^{\times 3})$ and $b_{1}\otimes b_{2} \in A^{\otimes 2}$, makes ${\rm EM}_{1}$ a left $A^{\otimes 2}$-module map. The chain map $\Gamma $ is defined by 
\[
\Gamma (v\otimes (u\otimes a)\otimes x)=(-1)^{|u||v|}(u\otimes va)\otimes x
\]
for $v\otimes (u\otimes a)\otimes x$ in ${\mathbb B} \otimes _{A^{\otimes 2}}(({\mathbb B}\otimes A) \otimes _{A^{\otimes 3}}A_{{\rm PL}}(LM^{\times 3}))$, and we also see that $\Gamma $ is an isomorphism with the inverse map $\Gamma ^{-1}((u \otimes v)\otimes x)=(-1)^{|u||v|} v\otimes (u\otimes 1)\otimes x$.\\
\indent
We now check the commutativity of the diagram \eqref{dlp}.
Since the composite ${\rm Comp} ({\rm Comp} \times 1)$ is homotopic to ${\rm Comp} '$, the square ${\rm (I)}$ commutes.
The equation ${\rm (I\hspace{-.1em}I)}:H({\rm Comp} \times 1){\rm EM}={\rm EM}_{2}H(1\otimes ({\rm Comp} \times 1)^{*} )$ is shown by the naturality of the Eilenberg-Moore map and the commutative diagram:
\[
\xymatrixrowsep{0.4cm}
\xymatrixcolsep{0.4cm}
\xymatrix{
& LM \times _{M} LM \times _{M} LM \ar[rr]^{ {\rm inc} } \ar[ld]_{ev_{0}} \ar[dd]^(0.3){{\rm Comp} \times 1 } & & ( LM \times _{M}LM )\times LM \ar[ld]_(0.6){ev_{0}\times ev_{0}} \ar[dd]_{{\rm Comp} \times 1}\\
M \ar[rr]^(0.3){\Delta} \ar[dd]_-{= }& & M^{\times 2} \ar[dd]^(0.3){=} \\
& LM\times _{M} LM \ar[rr]^(0.4){{\rm inc} } \ar[ld]_{ev_{0}} & & LM \times LM \ar[ld]^{ev_{0}^{\times 2}}\\
M\ar[rr]^{\Delta } & & M ^{\times 2}.
}
\]
A straightforward calculation shows a commutativity of the other squares ${\rm (I\hspace{-.1em}I\hspace{-.1em}I)}$, ${\rm (I\hspace{-.1em}V)}$ and ${\rm (V)}$, and also shows that the equality ${\rm EM}_{3} = {\rm EM}_{2}({\rm EM}_{1}\otimes 1)\Gamma ^{-1}$.
By the definition of dual loop products, 
\[
(Dlp \otimes 1)Dlp = H((\Delta^{!} \otimes 1) \otimes 1){\rm EM}_{1}^{-1}H({\rm Comp} \times 1) H(\Delta ^{!} \otimes 1){\rm EM}^{-1}H({\rm Comp})
\]
holds.
Therefore, the commutativity of the diagram \eqref{dlp} implies that
\[
(Dlp \otimes 1)Dlp = H((\Delta ^{!}\otimes \Delta^{!})\otimes 1){\rm EM}_{3}^{-1}H({\rm Comp} ').
\]
We next observe the composite $(1\otimes Dlp) Dlp$. Consider the commutative diagram
\[
\xymatrixrowsep{0.4cm}
\xymatrixcolsep{0.4cm}
\xymatrix{
& LM \times _{M} LM \times _{M} LM \ar[rrrr]^{{\rm inc}} \ar[ld]^{ev_{0}} \ar[dd]^(0.3){=} &&&& LM^{\times 3} \ar[ld]^{ev_{0}^{\times 3}} \ar[dd]^{=}\\
M \ar[rrrr]^{(1\times  \Delta )\Delta} \ar[dd]^-{=} &&&& M^{\times 3} \ar[dd]^(0.3){=}&\\
& LM \times _{M} LM \times _{M} LM \ar[rr]^{{\rm inc}} \ar[ld]^{ev_{0}} & & LM \times (LM \times _{M}LM) \ar[rr]^{{\rm inc}} \ar[ld]^{ev_{0}\times ev_{0}} & & LM^{\times 3}\ar[ld]^{ev_{0}^{\times 3}}\\
M \ar[rr]_{\Delta} & & M^{\times 2} \ar[rr]_-{1\times \Delta } & & M^{\times 3} & 
}
\]
where the two bottom squares and the top square are pull-back diagrams, and denote by
\begin{align*}
&{\rm EM}_{4} : (A\otimes {\mathbb B})\otimes _{A^{\otimes 3}} A_{{\rm PL}}(LM^{\times 3})\longrightarrow A_{{\rm PL}}(LM\times (LM\times _{M}LM)),\\
&{\rm EM}_{5} : {\mathbb B}\otimes _{A^{\otimes 2}}A_{{\rm PL}}(LM\times (LM\times _{M}LM))\longrightarrow A_{{\rm PL}}(LM\times _{M}LM\times _{M}LM) 
\end{align*}
the Eilenberg-Moore map of the two bottom squares.
Then, we have the commutative diagram:
\begin{equation} \label{dlp'}
\begin{split}
\xymatrixcolsep{1.2cm}
\xymatrixrowsep{0.7cm}
\xymatrix{
H^{*}(LM) \ar[d]_{H({\rm Comp})} \ar[r]^{=}  \ar@{}[rd]|{{\rm (I')}} & H^{*}(LM) \ar[d]^{H({\rm Comp}')} \\
H^{*}(LM\times _{M}LM) \ar[r]^-{H(1\times {\rm Comp})} \ar[d]_{\rm{EM}^{-1}}^{\cong } \ar@{}[rd]|{{\rm (I\hspace{-.1em}I')}}& H^{*}(LM \times _{M} LM \times _{M} LM) \ar[d]^{\rm{EM}_{5}^{-1}}_{\cong }\\
H^{*}({\mathbb B}\otimes _{A^{\otimes 2}}A_{{\rm PL}}(LM^{\times 2}))  \ar[d]_{H(\Delta^{!} \otimes 1)} \ar[r]^-{H(1\otimes (1\times {\rm Comp})^{*})} \ar@{}[rdd]|(0.3){{\rm (I\hspace{-.1em}I\hspace{-.1em}I')}} & H^{*}({\mathbb B} \otimes _{A^{\otimes 2}}A_{{\rm PL}}(LM \times (LM\times _{M} LM) ))  \ar[ldd]_{H(\Delta^{!} \otimes 1)} \ar[dd]^-{(1\otimes \rm{EM}_{4})^{-1}}_-{\cong } \ar@{}[lddd]|{{\rm (I\hspace{-.1em}V')}}\\
H^{*}(LM^{\times 2}) \ar[d]_{H(1 \times {\rm Comp})} & \\
H^{*}(LM \times (LM\times _{M}LM)) \ar[d]_{\rm{EM}_{4}^{-1}}^{\cong } & H^{*}({\mathbb B} \otimes _{A^{\otimes 2}}((A\otimes {\mathbb B}) \otimes _{A^{\otimes 3}}A_{{\rm PL}}(LM^{\times 3})))  \ar[dl]_{H(\Delta^{!} \otimes (1 \otimes 1 ) \otimes 1)} \ar[dd]^{\Gamma '}\\
H^{*}((A\otimes {\mathbb B}) \otimes _{A^{\otimes 3}}A_{{\rm PL}}(LM^{\times 3}))  \ar[d]_{H((1\otimes  \Delta ^{!}) \otimes 1)} \ar@{}[rd]|(0.4){{\rm (V')}} & \\
H^{*}(LM^{\times 3}) & H^{*}(({\mathbb B}\otimes _{A} {\mathbb B}) \otimes _{A^{\otimes 3}}A_{{\rm PL}}(LM^{\times 3})).  \ar[l]_-{(-1)^{d}H((\Delta ^{!}\otimes \Delta ^{!} ) \otimes 1)} 
}
\end{split}
\end{equation}
Remark that the left $A^{\otimes 2}$-module structure of $(A\otimes {\mathbb B}) \otimes _{A^{\otimes 3}}A_{{\rm PL}}(LM^{\times 3})$ given by the natural left $A^{\otimes 2}$-module structure of $A\otimes {\mathbb B}$, that is, 
\[
(b_{1}\otimes b_{2}) \cdot ((a\otimes u)\otimes x) = (-1)^{|a||b_{2}|}(b_{1}a\otimes b_{2}u)\otimes x
\]
for $(a\otimes u) \otimes x \in (A\otimes {\mathbb B})\otimes _{A^{\otimes 3}} A_{{\rm PL}}(LM^{\times 3}) $ and $b_{1}\otimes b_{2} \in A^{\otimes 2}$. It is readily seen that the map ${\rm EM}_{1}'$ preserves the left $A^{\otimes 2}$-module structure. The chain map $\Gamma '$ is defined by $\Gamma '( v\otimes (a\otimes u)\otimes x)=(-1)^{|a||v|} (av\otimes u)\otimes x$ for any $v \otimes (a\otimes u)\otimes x $ in ${\mathbb B} \otimes _{A^{\otimes 2}}((A\otimes {\mathbb B}) \otimes _{A^{\otimes 3}}A_{{\rm PL}}(LM^{\times 3}))$.\\
\indent
We can check the commutativity of the diagram \eqref{dlp'} as with the proof of the commutativity of the diagram \eqref{dlp} described above. Indeed, we see that
\[
(1\otimes Dlp )Dlp = H((1\otimes \Delta ^{!}) \otimes 1){\rm EM}_{4}^{-1}H(1\times {\rm Comp} )H(\Delta ^{!}\otimes 1){\rm EM}^{-1}H({\rm Comp} ),
\]
and the square ${\rm (I')}$ commutes since $(1\times {\rm Comp} ){\rm Comp}$ is homotopic to ${\rm Comp} '$. The equality ${\rm (II')}:H(1\times {\rm Comp}){\rm EM} = {\rm EM}_{5}H(1\otimes (1\times {\rm Comp} )^{*})$ is shown by the naturality of the Eilenberg-Moore map and the following commutative diagram:
\[
\xymatrixrowsep{0.4cm}
\xymatrixcolsep{0.4cm}
\xymatrix{
& LM \times _{M} LM \times _{M} LM \ar[rr]^{ {\rm inc} } \ar[ld]_-{ev_{0}} \ar[dd]^(0.3){1\times {\rm Comp}  } & & LM\times ( LM \times _{M}LM ) \ar[ld]^{ev_{0}\times ev_{0}} \ar[dd]^{1\times {\rm Comp} }\\
M \ar[rr]^(0.3){\Delta} \ar[dd]^-{= }& & M^{\times 2} \ar[dd]^(0.3){=} \\
& LM\times _{M} LM \ar[rr]^(0.4){{\rm inc} } \ar[ld]_{ev_{0}} & & LM \times LM \ar[ld]^{ev_{0}^{\times 2}}\\
M\ar[rr]^{\Delta } & & M ^{\times 2}.
}
\]
We also see that the commutativity of the diagrams ${\rm (I\hspace{-.1em}I\hspace{-.1em}I')}$, ${\rm (I\hspace{-.1em}V')}$, ${\rm (V)}$ and the equality ${\rm EM}_{3}={\rm EM}_{5}({\rm EM}_{4}\otimes 1)(\Gamma ')^{-1}$
by a straightforward calculation. Therefore, conclude that
\[
 (-1)^{d}(1\otimes Dlp)Dlp =H( ( \Delta ^{!}\otimes \Delta ^{!} ) \otimes 1){\rm EM}_{3}^{-1}H({\rm Comp} ')= (Dlp \otimes 1)Dlp.
\]
\end{proof}


\section{Dual loop coproduct and associativity}

We begin recalling the definition of the dual loop coproduct on Gorenstein spaces. Consider the pull-back diagram
\[
\xymatrix{
LM\times _{M} LM \ar[r]^-{{\rm Comp}} \ar[d]_{ev_{0}} & LM \ar[d]^{  j }\\
M \ar[r]^-{\Delta} & M \times M,
}
\]
where the map $j$ is defined by $j(\gamma ) = (\gamma (0), \gamma (\frac{1}{2}))$ for any loop $\gamma $ on $M$. It gives rise to the Eilenberg-Moore map
\[
{\rm EM}' : {\mathbb B}\otimes _{A^{\otimes 2}}A_{{\rm PL}}(LM)\longrightarrow A_{{\rm PL}}(LM \times _{M} LM)
\]
and we have the composition map called the {\it dual loop coproduct}
\[
\xymatrixrowsep{0.6cm}
\xymatrix{
Dlcop : H^{*}(LM\times LM) \ar[r]^-{H({\rm inc})} & H^{*}(LM\times _{M}LM) \ar[d]^-{({\rm EM'})^{-1}}_-{\cong } & \\
 & H^{*}({\mathbb B} \otimes _{A^{\otimes 2}}A_{{\rm PL}}(LM) ) \ar[r]^-{H(\Delta ^{!}\otimes 1)} & H^{*}(LM).
}
\]

\begin{prop}\label{dlcp}
The dual loop coproduct satisfy the identity
\[
Dlcop (Dlcop \otimes 1) = (-1)^{d}Dlcop (1\otimes Dlcop).
\]
\end{prop}

\begin{proof}
We first consider the following pull-back diagram:
\begin{equation}\label{tj}
\xymatrix{
LM\times _{M} LM \ar[r]^-{{\rm Comp}} \ar[d]_{ev_{0}} & LM \ar[d]^{ T_{M}\circ  j }\\
M \ar[r]^-{\Delta} & M \times M,
}
\end{equation}
where $T_{M}:M^{\times 2}\to M^{\times 2}$ is the switching map. Let $\tilde{{\rm EM}}'$ be the Eilenberg-Moore map of the diagram \eqref{tj}, $\tau :A^{\otimes 2}\to A^{\otimes 2}$ the switching map and $\tilde{\tau } : {\mathbb B}\to {\mathbb B}$ a chain map satisfy $\tilde{\tau }^{2}=1$ and makes the diagram commutes:
\[
\xymatrixrowsep{0.5cm}
\xymatrixcolsep{1.5cm}
\xymatrix{
{\mathbb B} \ar[r]^-{\tilde{\tau }} & {\mathbb B} \\
A^{\otimes 2} \ar[u]^-{\iota } \ar[r]^-{\tau} & A^{\otimes 2}, \ar[u]_-{\iota}
}
\]
where $\iota $ is the inclusion. Both the map $\Delta^{!}$ and $\tau \circ  \Delta^{!} \circ \tilde{\tau }$ are in ${\rm Ext}_{A^{\otimes 2}}^{d}(A,A^{\otimes 2})\cong {\mathbb Q}$, there is a scalar $\lambda   \in {\mathbb Q}$ such that $\Delta^{!}=\lambda  \tau \circ  \Delta^{!} \circ \tilde{\tau }$ in ${\rm Ext}_{A^{\otimes 2}}^{d}(A,A^{\otimes 2})$. Hence, we obtain the following commutative diagram:
\[
\xymatrixrowsep{0.65cm}
\xymatrixcolsep{1.5cm}
\xymatrix{
H^{*}(LM\times LM) \ar[d]_-{H({\rm inc} )} \ar@/_30mm/[ddd]_{Dlcop}& \\
H^{*}(LM\times _{M}LM) \ar[d]_-{({\rm EM}')^{-1}}^-{\cong } \ar[rd]^-{(\tilde{{\rm EM}}')^{-1}}_-{\cong } & \\
H^{*}({\mathbb B}\otimes _{A^{\otimes 2}}A_{{\rm PL}}(LM)) \ar[r]^-{\tilde{\tau}\otimes _{\tau}1 } \ar[d]_-{H(\Delta^{!}\otimes 1)} & H^{*}({\mathbb B}\otimes _{A^{\otimes 2}}A_{{\rm PL}}(LM)) \ar[ld]^-{\lambda H(\Delta^{!}\otimes 1)}\\
H^{*}(LM\times LM).  & 
}
\]
Here, the commutativity of the upper right triangle is shown by the commutative diagram:
\[
\xymatrixrowsep{0.4cm}
\xymatrixcolsep{0.4cm}
\xymatrix{
& LM \times _{M} LM  \ar[rr]^{{\rm Comp} } \ar[ld]_-{ev_{0}} \ar[dd]^(0.3){=  } & & LM \ar[ld]^{j} \ar[dd]^{= }\\
M \ar[rr]^(0.3){\Delta} \ar[dd]^-{= }& & M^{\times 2} \ar[dd]^(0.3){T_{M}} \\
& LM\times _{M} LM \ar[rr]^(0.4){{\rm Comp} } \ar[ld]_{ev_{0}} & & LM  \ar[ld]^{T_{M}\circ j}\\
M\ar[rr]^{\Delta } & & M ^{\times 2}.
}
\]
Since
$\Delta ^{!} = \tau \tau \Delta^{!}  \tilde{\tau } \tilde{\tau}  = \lambda \tau \Delta^{!}  \tilde{\tau } =\lambda ^{2} \Delta^{!}$,
we see that $\lambda ^{2}=1$. Let ${\rm EM}_{1}'$ and $\tilde{{\rm EM}}_{1}'$ be the Eilenberg-Moore maps of the top and bottom square of the following diagram, respectively:
\begin{equation}\label{dlcop1}
\xymatrixrowsep{0.4cm}
\xymatrixcolsep{0.4cm}
\xymatrix{
& (LM \times _{M} LM) \times  LM \ar[rr]^{ {\rm Comp} \times 1 } \ar[ld]_-{ev_{0}\times ev_{0}} \ar[dd]^(0.3){=  } & & LM\times LM  \ar[ld]^{j \times ev_{0}} \ar[dd]^{ = }\\
M^{\times 2} \ar[rr]^(0.3){\Delta \times 1} \ar[dd]^-{= }& & M^{\times 3} \ar[dd]^(0.3){T_{M}\times 1} \\
& (LM \times _{M} LM) \times  LM \ar[rr]^{ {\rm Comp} \times 1 } \ar[ld]_-{ev_{0}\times ev_{0}}  & & LM\times LM  \ar[ld]^{T_{M}\circ j \times ev_{0}} \\
M^{\times 2} \ar[rr]^(0.3){\Delta \times 1} & & M^{\times 3}.
}
\end{equation}

Then, the definition of the dual loop coproduct and the equation $\lambda ^{2}=1$ show that
\begin{align*}
Dlcop (Dlcop \otimes 1) &= H(\Delta ^{!}\otimes 1)({\rm EM}')^{-1}H({\rm inc}) H ((\Delta ^{!} \otimes 1) \otimes 1)({\rm EM}_{1}')^{-1}H({\rm inc} \times 1)\\
                                &=\lambda ^{2}H(\Delta ^{!}\otimes 1)(\tilde{{\rm EM}}')^{-1}H({\rm inc}) H ((\Delta ^{!} \otimes 1) \otimes 1)(\tilde{{\rm EM}}_{1}')^{-1}H({\rm inc} \times 1)\\
                                &=H(\Delta ^{!}\otimes 1)(\tilde{{\rm EM}}')^{-1}H({\rm inc}) H ((\Delta ^{!} \otimes 1) \otimes 1)(\tilde{{\rm EM}}_{1}')^{-1}H({\rm inc} \times 1).
\end{align*}

Consider the commutative diagram
\begin{equation}\label{dlcop2}
\xymatrixrowsep{0.4cm}
\xymatrixcolsep{0.4cm}
\xymatrix{
& LM \times _{M} LM \times _{M} LM \ar[rrrr]^{{\rm Comp}'} \ar[ld]_{ev_{0}}  &&&& LM \ar[ld]_{j'}\\
M \ar[rrrr]^{\Delta}  &&&& M^{\times 3}&\\
& LM \times _{M} LM \times _{M} LM \ar[rr]^{{\rm Comp}\times 1} \ar[ld]^{ev_{0}} \ar[uu]^(0.3){=}& & LM \times _{M}LM \ar[rr]^(0.45){{\rm Comp}} \ar[ld]^{j_{2}} & & LM \ar[ld]^{j_{1}} \ar[uu]_{\xi _{1}}\\
M \ar[rr]_{\Delta} \ar[uu]^-{=}& & M \times M \ar[rr]_-{1\times \Delta } & & M^{\times 3},  \ar[uu]^(0.3){=}& 
}
\end{equation}
where the maps $j'$, $j_{1}$ and $j_{2}$ are defined by
\[
j' (\gamma ) = (\gamma \Bigl( \frac{1}{3} \Bigr), \gamma (0), \gamma \Bigl( \frac{2}{3} \Bigr) ), \ j_{1}=(\gamma \Bigl( \frac{1}{4} \Bigr), \gamma (0), \gamma \Bigl( \frac{1}{2} \Bigr) ), \ j_{2}(\gamma _{1},\gamma _{2})=(\gamma _{1} \Bigl( \frac{1}{2} \Bigr) , \gamma _{1}(0))
\]
for any $\gamma \in LM$ and $(\gamma _{1},\gamma _{2})\in LM\times _{M}LM$, respectively. The map $\xi _{1}$ is given by
\begin{eqnarray*}
\xi_{1} (\gamma )(t) = \left\{ \begin{array}{ll}
 \gamma (\frac{3}{4}t) & (0\leq t\leq \frac{2}{3}) \\
 \gamma (\frac{3}{2}t-\frac{1}{2}) & (\frac{2}{3}\leq t\leq 1)
 \end{array} \right.
 \end{eqnarray*}
for $\gamma \in LM$ and $t\in [0,1]$. Since the bottom two squares and the top square of the diagram \eqref{dlcop2} are pull-back diagrams, it give rise to the Eilenberg-Moore maps
\begin{align*}
&{\rm EM}'_{2}: {\mathbb B}\otimes _{A^{\otimes 2}}A_{{\rm PL}}(LM \times _{M}LM)\longrightarrow A_{{\rm PL}}(LM\times _{M}LM\times _{M}LM),\\
&{\rm EM}'_{3}: (A\otimes {\mathbb B})\otimes _{A^{\otimes 3}}A_{{\rm PL}}(LM)\longrightarrow A_{{\rm PL}}(LM\times _{M} LM),\\
&{\rm EM}'_{4}:({\mathbb B}\otimes _{A}{\mathbb B})\otimes _{A^{\otimes 3}}A_{{\rm PL}}(LM)\longrightarrow A_{{\rm PL}}(LM\times _{M}LM\times _{M}LM),
\end{align*}
respectively. We here note that there is a map between the bottom and left-hand side square of the diagram \eqref{dlcop2} and the diagram \eqref{dlcop1} such as the following:
\begin{equation}\label{dlcop3}
\xymatrixrowsep{0.45cm}
\xymatrixcolsep{0.7cm}
\xymatrix{
& LM \times _{M} LM \times _{M} LM \ar[rr]^-{ {\rm Comp} \times 1} \ar[ld]_{ev_{0}} \ar[dd]^(0.3){{\rm inc} } & & LM \times _{M}LM  \ar[ld]^{j_{2}} \ar[dd]^{{\rm inc}}\\
M \ar[rr]^(0.3){\Delta} \ar[dd]^{\Delta }& & M \times M  \ar[dd]^(0.7){1\times \Delta } \\
& (LM\times _{M} LM) \times LM \ar[rr]^(0.45){{\rm Comp} \times 1} \ar[ld]_{ev_{0}\times ev_{0}} & & LM \times LM \ar[ld]^{T_{M}\circ j\times ev_{0}}\\
M \times M \ar[rr]_{\Delta \times 1} & & M ^{\times 3},
}
\end{equation}
We hence get the following commutative diagram:
\begin{equation}\label{dlcop4}
\xymatrixrowsep{0.7cm}
\xymatrix{
H^{*}(LM\times LM \times LM) \ar[d]_{H({\rm inc} \times 1)}   \ar[rd]^{H({\rm inc})} \ar@{}[rddd]|(0.2){{\rm (I)}}& 
\\
H^{*}((LM\times _{M}LM)\times LM) \ar[r]^{H({\rm inc})}    \ar[d]_{(\tilde{\rm{EM}}'_{1})^{-1}}^{\cong }  \ar@{}[rd]|{{\rm (I\hspace{-.1em}I)}}  & H^{*}(LM\times _{M} LM \times _{M} LM)  \ar[d]^{({\rm EM}'_{2})^{-1}}_{\cong }\\
H^{*}(({\mathbb B}\otimes A)\otimes _{A^{\otimes 3}}A_{{\rm PL}}(LM\times LM))\ar[d]_{H((\Delta^{!} \otimes 1 ) \otimes 1)} \ar[r]^-{ \Delta ^{*} \otimes _{(1\times \Delta)^{*}} {\rm inc}^{*} } \ar@{}[rddd]|(0.2){{\rm (I\hspace{-.1em}I\hspace{-.1em}I)}}& H^{*}({\mathbb B}\otimes _{A^{\otimes 2}}A_{{\rm PL}}(LM\times _{M}LM))  \ar[ldd]_{H(\Delta^{!} \otimes 1)} \ar[ddd]^-{1\otimes ( \rm{EM}'_{3})^{-1}}_-{\cong } \\
H^{*}(LM\times LM) \ar[d]_{H({\rm inc})} & \ar@{}[ldd]|(0.5){{\rm (I\hspace{-.1em}V)}}\\
H^{*}(LM\times _{M}LM) \ar[d]_{(\tilde{\rm{EM}}')^{-1}}^{\cong}& \\
H^{*}({\mathbb B}\otimes _{A^{\otimes 2}}A_{{\rm PL}}(LM))  \ar[d]_{H(\Delta^{!} \otimes 1)} & H^{*}({\mathbb B} \otimes _{A^{\otimes 2}}((A\otimes {\mathbb B})\otimes _{A^{\otimes 3}}A_{{\rm PL}}(LM)))  \ar[l]_-{\Theta   _{1}'} \ar@{}[ld]_-{{\rm (V)}}\\
H^{*}(LM) & H^{*}(({\mathbb B}\otimes _{A}{\mathbb B})\otimes _{A^{\otimes 3}}A_{{\rm PL}}(LM))\ar[u]_{\Theta   _{1}} \ar[l]_-{(-1)^{d}H(\xi _{1})H((\Delta^{!} \otimes \Delta^{!}) \otimes 1)}
}
\end{equation}
Here, $\Delta ^{*}\otimes _{(1\times \Delta ) ^{*}}{\rm inc}^{*}$ is
given by
\[
( \Delta ^{*}\otimes _{(1\times \Delta ) ^{*}}{\rm inc}^{*})( (u\otimes a) \otimes x) = ua\otimes {\rm inc}^{*}(x) 
\]
for $ (u\otimes a)\otimes x \in ({\mathbb B}\otimes A) \otimes _{A^{\otimes 3}}A_{{\rm PL}}(LM\times LM)$. The chain maps $\Theta _{1}$ and $\Theta _{1}'$ are defined as follows:
\begin{align*}
& \Theta _{1}((u\otimes v)\otimes x) = u \otimes (1\otimes v) \otimes \xi _{1}^{*}(x) ,\\
& \Theta _{1}'(v \otimes (a\otimes u) \otimes x ) = \sum (-1)^{|v_{2}||a|+(|v_{1}|+|a|)(|v_{2}|+|u|)}  v_{2}u\otimes \Bigl( ev_{\frac{1}{4}}^{*}(v_{1}a)\cdot x \Bigr),
\end{align*}
where $\Delta ^{!}(v)=\sum v_{1}\otimes v_{2}$. We also note that the left $A^{\otimes 2}$-module structure of $(A\otimes {\mathbb B})\otimes _{A^{\otimes 3}}A_{{\rm PL}}(LM)$ given by the natural left $A^{\otimes 2}$-module structure of $A\otimes {\mathbb B}$ makes ${\rm EM}_{3}'$ a left $A^{\otimes 2}$-module map.\\
\indent
We now check the commutativity of the diagram \eqref{dlcop4}. A commutativity of ${\rm (I)}$ is trivial, and ${\rm (I\hspace{-.1em}I)} $ is shown by the naturality of the Eilenberg-Moore map and the commutative diagram \eqref{dlcop3}. It is readily seen that the diagram ${\rm (I\hspace{-.1em}I\hspace{-.1em}I)}$, ${\rm (I\hspace{-.1em}V)}$ and ${\rm (V)}$.
By the diagram \eqref{dlcop2}, we see that ${\rm EM}'_{4}={\rm EM}'_{2}(1\otimes {\rm EM}'_{3})\Theta _{1}$. Therefore,
\[
Dlcop (Dlcop \otimes 1)= (-1)^{d}H(\xi _{1})H( ( \Delta ^{!} \otimes \Delta ^{!} ) \otimes 1)({\rm EM'}_{4})^{-1}H({\rm inc}).
\]
On the other hand, consider the two commutative diagram
\begin{equation}\label{dlcop'1}
\xymatrixrowsep{0.6cm}
\xymatrixcolsep{0.7cm}
\xymatrix{
& LM \times _{M} LM \times _{M} LM \ar[rr]^{1\times  {\rm Comp}} \ar[ld]_{ev_{0}} \ar[dd]^(0.3){{\rm inc} } & & LM \times _{M}LM  \ar[ld]^{j'_{2}} \ar[dd]^{{\rm inc}}\\
M \ar[rr]^(0.3){\Delta} \ar[dd]^{\Delta }& & M \times M  \ar[dd]^(0.3){\Delta \times 1} \\
& LM \times (LM\times _{M} LM) \ar[rr]^(0.45){1\times {\rm Comp}} \ar[ld]_{ev_{0}\times ev_{0}} & & LM \times LM \ar[ld]^{ ev_{0} \times j}\\
M \times M \ar[rr]_{1\times \Delta } & & M ^{\times 3}
}
\end{equation}
and
\begin{equation}\label{dlcop'2}
\xymatrixrowsep{0.6cm}
\xymatrixcolsep{0.2cm}
\xymatrix{
& LM \times _{M} LM \times _{M} LM \ar[rrrr]^{{\rm Comp}'} \ar[ld]^{ev_{0}}  &&&& LM \ar[ld]^{j'} \\
M \ar[rrrr]^{\Delta}  &&&& M^{\times 3} &\\
& LM \times _{M} LM \times _{M} LM \ar[rr]^-{1\times {\rm Comp}} \ar[ld]^{ev_{0}} \ar[uu]^(0.3){=}& & LM \times _{M}LM \ar[rr]^{{\rm Comp}} \ar[ld]^{j_{2}'} & & LM \ar[ld]^{j'_{1}} \ar[uu]_{\xi _{2}}\\
M \ar[rr]_{\Delta}\ar[uu]^-{=} & & M \times M \ar[rr]_-{\Delta \times 1} & & M^{\times 3}, \ar[uu]^(0.3){=}& 
}
\end{equation}
where $j_{1}'$ and $j_{2}'$ are defined by $j_{1}'(\gamma ) = (\gamma ( \frac{1}{2}) , \gamma (0),  \gamma ( \frac{3}{4} ) )$, $j_{2}'(\gamma _{1},\gamma _{2})= (\gamma _{2}(0), \gamma _{2}( \frac{1}{2}) )$
for $\gamma \in LM$ and $(\gamma _{1},\gamma _{2})\in LM \times _{M}LM$, respectively. The map $\xi _{2}$ is given by
\begin{eqnarray*}
\xi_{2}(\gamma )(t) =
\left\{
\begin{array}{ll}
\gamma (\frac{3}{2}t) & (0\leq t\leq \frac{1}{3}) \\
\gamma (\frac{3}{4}t+\frac{1}{4}) & (\frac{1}{3}\leq t\leq 1).
\end{array}
\right.
\end{eqnarray*}
We denote by
\begin{align*}
&{\rm EM}'_{5}: (A\otimes {\mathbb B})\otimes _{A^{\otimes 3}}A_{{\rm PL}}(LM \times LM)\longrightarrow A_{{\rm PL}}(LM \times (LM\times _{M}LM)),\\
&{\rm EM}'_{6}: {\mathbb B}\otimes _{A^{\otimes 2}}A_{{\rm PL}}(LM \times_{M} LM)\longrightarrow A_{{\rm PL}}(LM \times_{M} LM \times _{M} LM)\\
&{\rm EM}'_{7}: ({\mathbb B}\otimes A)\otimes _{A^{\otimes 3}}A_{{\rm PL}}(LM )\longrightarrow A_{{\rm PL}}(LM \times_{M} LM)
\end{align*}
the Eilenberg-Moore maps of the bottom square of the diagram \eqref{dlcop'1}, the top square of \eqref{dlcop'1} and the bottom and right-hand side square of \eqref{dlcop'2}, respectively. Then, we have the commutative diagram:
\begin{equation}\label{dlcop'3}
\xymatrixrowsep{0.7cm}
\xymatrix{
H^{*}(LM\times LM \times LM) \ar[d]_{H(1\times {\rm inc} )}  \ar[rd]^{H({\rm inc})} \ar@{}[rddd]|(0.2){{\rm (I')}}&
\\
H^{*}(LM\times (LM\times _{M}LM)) \ar[r]^{H({\rm inc})}  \ar[d]_{( \rm{EM}'_{5})^{-1}}^{\cong } \ar@{}[rd]|{{\rm (I\hspace{-.1em}I')}} & H^{*}(LM\times _{M} LM \times _{M} LM) \ar[d]^{( \rm{EM}'_{6})^{-1} }_{\cong }\\
H^{*}((A\otimes {\mathbb B})\otimes _{A^{\otimes 3}}A_{{\rm PL}}(LM\times LM))  \ar[d]_{H(( 1\otimes \Delta^{!}) \otimes 1)} \ar[r]^{\Delta^{*} \otimes _{(\Delta \times 1)^{*}} {\rm inc}^{*}} \ar@{}[rddd]|(0.2){{\rm (I\hspace{-.1em}I\hspace{-.1em}I')}}& H^{*}({\mathbb B}\otimes _{A^{\otimes 2}}A_{{\rm PL}}(LM\times _{M}LM))  \ar[ldd]_{H(\Delta^{!} \otimes 1)}  \ar[ddd]^-{1\otimes  (\rm{EM}'_{7})^{-1}}_-{\cong}\\
H^{*}(LM\times LM) \ar[d]_{H({\rm inc})} & \ar@{}[ldd]|(0.5){{\rm (I\hspace{-.1em}V')}}\\
H^{*}(LM\times _{M}LM) \ar[d]_{( \rm{EM}')^{-1}}^{\cong } & \\
H^{*}({\mathbb B}\otimes _{A^{\otimes 2}}A_{{\rm PL}}(LM))  \ar[d]_{H(\Delta^{!} \otimes 1)} & H^{*}({\mathbb B}\otimes _{A^{\otimes 2}}(({\mathbb B}\otimes A)\otimes _{A^{\otimes 3}}A_{{\rm PL}}(LM))) \ar[l]_-{\Theta _{2}'} \ar@{}[ld]_-{{\rm (V')}} \\
H^{*}(LM) & H^{*}(({\mathbb B}\otimes _{A}{\mathbb B})\otimes _{A^{\otimes 3}}A_{{\rm PL}}(LM)).  \ar[u]_{\Theta _{2}} \ar[l]_-{H(\xi_{2})H( ( \Delta^{!} \otimes \Delta^{!}) \otimes 1)}
}
\end{equation}
Here, $\Delta^{*}\otimes _{(\Delta \times 1)^{*}}{\rm inc}^{*}$ is the map induced by the map between the pull-back diagrams described in \eqref{dlcop'2}, that is,
\[
(\Delta^{*}\otimes _{(\Delta \times 1)^{*}}{\rm inc}^{*} ) ( (a\otimes u)\otimes x) =  au \otimes {\rm inc}^{*}(x)
\]
for $ (a\otimes u) \otimes x\in (A\otimes {\mathbb B})\otimes _{A^{\otimes 3}}A_{{\rm PL}}(LM\times LM)$. The chain maps $\Theta _{2}$ and $\Theta _{2}'$ are given by
\begin{align*}
&\Theta _{2}( (u\otimes v) \otimes x )= (-1)^{|u||v|} v\otimes (u\otimes 1)\otimes \xi _{2}^{*}(x),\\
&\Theta '_{2}(v\otimes (u\otimes a )\otimes x) = \sum (-1)^{|u||v_{1}|+(|u|+|a|)|v_{2}|} uv_{1}\otimes \Bigl( ev^{*}_{\frac{3}{4}}\rho (av_{2})\cdot x \Bigr),
\end{align*}
where $\Delta ^{!}(v)= \sum v_{1}\otimes v_{2}$. By the definition of the dual loop coproduct, we see that
\[
Dlcop (1\otimes Dlcop ) = H( \Delta ^{!} \otimes 1)({\rm EM}')^{-1}H({\rm inc}) H ((1\otimes \Delta ^{!})\otimes 1)({\rm EM}_{5}')^{-1}H(1\times {\rm inc} )
\]
Similar argument of the proof of the commutativity of the diagram \eqref{dlcop4} described above shows the commutativity of the diagram \eqref{dlcop'3}. We also see that the equation ${\rm EM}'_{4}=\Theta _{2}(1\otimes {\rm EM}'_{7}){\rm EM}'_{6}$ holds,
therefore the commutativity of the diagram \eqref{dlcop'3} shows
\[
Dlcop (1\otimes Dlcop) = H(\xi _{2} )H(( \Delta ^{!} \otimes \Delta ^{!})\otimes 1){\rm EM'}_{4}^{-1}H({\rm inc}).
\]
Since the map $H : LM \times [0,1] \to LM $ defined by
\begin{eqnarray*}
H(\gamma ,s)(t) = \left\{ \begin{array}{ll}
 \gamma (\frac{3}{4}(1+s)t) & (0\leq t\leq \frac{1}{3}) \\
 \gamma (\frac{3}{4}t+\frac{1}{4}s) & (\frac{1}{3} \leq t \leq \frac{2}{3})\\
 \gamma ((\frac{3}{2}-\frac{3}{4}s)t-\frac{1}{2}+\frac{3}{4}s) & (\frac{2}{3}\leq t\leq 1),
 \end{array} \right.
 \end{eqnarray*}
for $\gamma \in LM$ and $s,t \in [0,1]$ is a homotopy from $\xi_{1}$ to $\xi_{2}$, we have
\[
Dlcop (Dlcop \otimes 1) = (-1)^{d}Dlcop (1\otimes Dlcop).
\]
This completes the proof.
\end{proof}


\section{Frobenius compatibility}
In this section, we prove that the dual loop product and the dual loop coproduct satisfy the Frobenius compatibility.
\begin{prop}\label{Fro}
One has
\[
(-1)^{d} (1\otimes Dlcop ) (Dlp \otimes 1)= Dlp \circ  Dlcop = (-1)^{d} ( Dlcop \otimes 1) (1\otimes Dlp ) . \]
\end{prop}

\begin{proof}
We note that the notations ${\rm EM}_{i}$ and ${\rm EM'}_{i}$ are described in \S 2 and \S 3, respectively. Consider the following three commutative diagrams:
\begin{equation}\label{fro1}
\xymatrixrowsep{0.4cm}
\xymatrixcolsep{0.9cm}
\xymatrix{
& LM \times _{M} LM \times _{M} LM \ar[rr]^-{{\rm Comp} \times 1} \ar[dd]^(0.3){({\rm Comp} \times 1)T_{(123)}} \ar[ld]_{ev_{0}}  & & LM \times _{M} LM \ar[dd]^{\xi {\rm Comp}}  \ar[ld]^{(ev_{1,0},ev_{1,\frac{1}{2}})} \\
M \ar[dd]^{=} \ar[rr]^(0.3){\Delta} & & M^{\times 2} \ar[dd]^(0.3){=} & \\
& LM \times _{M} LM \ar[ld]_{ev_{0}}  \ar[rr]^(0.45){{\rm Comp} } & & LM \ar[ld]^{j} \\
M \ar[rr]^{\Delta }& & M^{\times 2} &
}
\end{equation}

\begin{equation}\label{fro2}
\xymatrixrowsep{0.4cm}
\xymatrixcolsep{0.9cm}
\xymatrix{
& LM \times _{M} LM \times _{M} LM \ar[rr]^{{\rm inc}} \ar[dd]^(0.3){{\rm inc}} \ar[ld]_{ev_{0}}  & & LM \times (LM \times _{M} LM )\ar[dd]^{{\rm inc}}  \ar[ld]_{ev_{0}\times ev_{0}} \\
M \ar[dd]^-{\Delta } \ar[rr]^(0.3){\Delta} & & M^{\times 2} \ar[dd]^(0.3){1\times \Delta } & \\
& ( LM \times _{M} LM) \times LM \ar[ld]^{ev_{0}\times ev_{0}}  \ar[rr]^(0.3){{\rm inc} } & & LM^{\times 3}\ar[ld]^{ev_{0}^{\times 3}} \\
M \times M \ar[rr]^{\Delta \times 1}& & M^{\times 3} &
}
\end{equation}
\begin{equation}\label{fro3}
\xymatrixrowsep{0.4cm}
\xymatrixcolsep{0.3cm}
\xymatrix{
& LM \times _{M} LM \times _{M} LM \ar[rr]^-{{\rm inc}} \ar[ld]_{ev_{0}} \ar[dd]^(0.3){T_{(321)}} & & LM \times (LM\times _{M} LM) \ar[rr]^-{1\times {\rm Comp}} \ar[ld]_{ev_{0}\times ev_{0}} & & LM^{\times 2}  \ar[ld]_{ev_{0}\times j} \ar[dd]^{=} \\
M \ar[rr]^(0.3){\Delta } \ar[dd]^{=} & & M ^{\times 2} \ar[rr]^{1\times \Delta } & & M^{\times 3} \ar[dd]^(0.3){=} &\\
& LM \times _{M} LM \times _{M} LM \ar[rr]^{{\rm Comp} \times 1}  \ar[ld]^{ev_{0}} & & LM \times _{M} LM \ar[rr]^{({\rm inc} ) \circ  T} \ar[ld]_(0.7){(ev_{1,0},ev_{1,\frac{1}{2}})}  & & LM^{\times 2} \ar[ld]^{ev_{0}\times j} \\
M \ar[rr]^{\Delta } & & M ^{\times 2} \ar[rr]^{\Delta \times 1} & & M^{\times 3} &
}
\end{equation}
Here $T:LM\times _{M} LM\to LM\times _{M}LM $ is the switching map and  $\xi : LM\to LM$ is a map given by
\begin{eqnarray*}
\xi (\gamma )(t) =
\left\{
\begin{array}{ll}
\gamma (2t + \frac{1}{2}) & (0\leq t\leq \frac{1}{4}) \\
\gamma (t- \frac{1}{4}) & (\frac{1}{4}\leq t\leq \frac{1}{2})\\
\gamma (\frac{1}{2}t) & (\frac{1}{2}\leq t\leq 1)
\end{array}
\right.
\end{eqnarray*}
for $\gamma \in LM$ and $t\in [0,1]$, and  the maps $T_{(123)}$ and $T_{(321)}$ are given by
\[
T_{(123)}(\gamma _{1},\gamma _{2},\gamma _{3})= (\gamma _{3},\gamma _{1},\gamma _{2}), \ T_{(321)}(\gamma _{1},\gamma _{2},\gamma _{3})= (\gamma _{2},\gamma _{3},\gamma _{1})
\]
for $(\gamma _{1},\gamma _{2},\gamma _{3})$ in $LM \times_{M} LM \times _{M} LM $.
Denote by
\begin{align*}
& {\rm EM}_{1}'' : {\mathbb B}\otimes _{A^{\otimes 2}} A_{{\rm PL}}(LM \times _{M} LM ) \longrightarrow A_{{\rm PL}}(LM \times _{M} LM \times_{M} LM)\\
& {\rm EM}_{2}'' : ( {\mathbb B} \otimes A) \otimes _{A^{\otimes 3}} A_{{\rm PL}}(LM ^{\times 2} ) \longrightarrow A_{{\rm PL}}(LM \times _{M} LM )
\end{align*}
the Eilenberg-Moore maps of the top square of \eqref{fro1} and the bottom and right-hand side square of \eqref{fro3}, respectively.
In order to obtain the result, we consider the following diagrams:

\begin{landscape}
\begin{equation}\label{fro4}
\xymatrixrowsep{0.8cm}
\xymatrixcolsep{0.5cm}
\xymatrix{
H^{*}(LM \times LM) \ar[r]_-{H({\rm Comp} \times 1)} \ar[dd]^{H({\rm inc})} 
 \ar@/^6mm/[rrr]^{(Dlp\otimes 1)} \ar@{}[rdd]|{{\rm (I)}} & H^{*}((LM ^{\times _{M}2})\times LM) \ar[dd]^{H({\rm inc})} \ar@{}[rdd]|{{\rm (I\hspace{-.1em}I)}}  \ar[r]_-{({\rm EM}_{1})^{-1}}^{\cong }& H^{*}(({\mathbb B}\otimes A)\otimes _{A^{\otimes 3}}A_{{\rm PL}}(LM^{\times 3})) \ar[dd]^{\Delta ^{*}\otimes _{(1\times \Delta )^{*}}{\rm inc}^{*}} \ar[r]_-{H( (\Delta ^{!}\otimes 1)\otimes 1)} \ar@{}[rdd]|{{\rm (I\hspace{-.1em}I\hspace{-.1em}I)}}& H^{*}(LM^{\times 3}) \ar[dd]_{H(1\times {\rm inc})}
\\
&&&\\
H^{*}(LM \times _{M} LM) \ar[r]^-{H({\rm Comp} \times _{M}1)}  \ar@{}[rdddd]|{{\rm (I\hspace{-.1em}V)}} \ar[dddd]^{({\rm EM'})^{-1}}_{\cong } & H^{*}(LM ^{\times _{M}3}) \ar@{}[rdddddd]|(0.35){{\rm (V)}} \ar[r]^-{({\rm EM}_{5})^{-1}}_-{\cong } & H^{*}({\mathbb B}\otimes _{A^{\otimes 2}}A_{{\rm PL}}(LM \times (LM ^{\times _{M}2})))  \ar[r]^-{H(\Delta ^{!} \otimes 1)} \ar@{}[rdd]|{{\rm (V\hspace{-.1em}I)}} \ar[dd]_{(1\otimes {\rm EM'}_{5})^{-1}}^{\cong } & H^{*}(LM \times (LM ^{\times _{M}2})) \ar[dd]_{({\rm EM}'_{5})^{-1}}^{\cong } \\
&&&\\
 & H^{*}(LM ^{\times _{M}3}) \ar[uu]^-{\cong }_-{H(T_{(321)})=H(T_{(123)})^{-1}} \ar[dd]^{({\rm EM}''_{1})^{-1}}_{\cong } & H^{*}({\mathbb B}\otimes _{A^{\otimes 2}}((A\otimes {\mathbb B})\otimes _{A^{\otimes 3}}A_{{\rm PL}}(LM ^{\times 2})))  \ar[r]^-{H(\Delta ^{!}\otimes (1\otimes 1)\otimes 1)}  \ar@{}[rdddd]^{{\rm (I\hspace{-.1em}X)}} & H^{*}((A\otimes {\mathbb B})\otimes _{A^{\otimes 3}}A_{{\rm PL}}(LM ^{\times 2}))  \ar[dddd]_{H( (1\otimes \Delta ^{!})\otimes 1)} \\
 &&&\\
H^{*}({\mathbb B}\otimes _{A^{\otimes 2}}A_{{\rm PL}}(LM))  \ar[r]^{1\otimes (\xi {\rm Comp} )^{*} } \ar[dd]^{H(\Delta ^{!} \otimes 1)}  \ar@{}[rdd]|{{\rm (V\hspace{-.1em}I\hspace{-.1em}I)}}  & H^{*}({\mathbb B}\otimes _{A^{\otimes 2}}A_{{\rm PL}}(LM ^{\times _{M}2}))  \ar[dd]^{H(T) H(\Delta ^{!}\otimes 1)}  \ar@{}[rdd]|{{\rm (V\hspace{-.1em}I\hspace{-.1em}I\hspace{-.1em}I)}} \ar[r]^-{(1\otimes {\rm EM}''_{2})^{-1}}_-{\cong } & H^{*}({\mathbb B} \otimes _{A^{\otimes 2}}(({\mathbb B}\otimes A) \otimes _{A^{\otimes 3}}A_{{\rm PL}}(LM^{\times 2})))   \ar[uu]^{\cong }_{\Phi  }  \ar[dd]^{\Phi '}  & \\
&&&\\
H^{*}(LM) \ar[r]^{H({\rm Comp})} \ar@/_6mm/[rrr]^{Dlp} & H^{*}(LM ^{\times _{M}2})  \ar[r]^{({\rm EM})^{-1}}_{\cong } & H^{*}({\mathbb B}\otimes _{A^{\otimes 2}}A_{{\rm PL}}(LM^{\times 2})) \ar[r]^{(-1)^{d}H(\Delta ^{!}\otimes 1)} & H^{*}(LM\times LM).
}
\end{equation}
\end{landscape}
Remark that the left $A^{\otimes 2}$-module structure of $(A\otimes {\mathbb B})\otimes _{A^{\otimes 3}}A_{{\rm PL}}(LM^{\times 2})$ given by the natural left $A^{\otimes 2}$-module structure of $A\otimes {\mathbb B}$ makes ${\rm EM}_{5}$ a left $A^{\otimes 2}$-module map. Similarly, the left $A^{\otimes 2}$-module structure of $( {\mathbb B}\otimes A)\otimes _{A^{\otimes 3}}A_{{\rm PL}}(LM^{\times 2})$ given by the natural right $A^{\otimes 2}$-module structure of ${\mathbb B}\otimes A$ makes ${\rm EM}_{2}''$ a left $A^{\otimes 2}$-module map.
The chain map $\Phi  $ is defined by
\[
\Phi  (v \otimes (u \otimes a) \otimes x) = (-1)^{|u||v|} u \otimes (1 \otimes va ) \otimes x
\]
for $v \otimes (u \otimes a) \otimes x \in {\mathbb B} \otimes _{A^{\otimes 2}}(({\mathbb B}\otimes A) \otimes _{A^{\otimes 3}}A_{{\rm PL}}(LM^{\times 2}))$.
We also see that $\Phi  $ is an isomorphism with the inverse map
\[
\Phi  ^{-1}(v\otimes (a\otimes u)\otimes x) = (-1)^{|u|(|a|+|v|)+|a||v|}u\otimes (av\otimes 1) \otimes x
\]
for $v\otimes (a\otimes u)\otimes x$ in ${\mathbb B} \otimes _{A^{\otimes 2}}((A\otimes {\mathbb B}) \otimes _{A^{\otimes 3}}A_{{\rm PL}}(LM^{\times 2}))$. The chain map $\Phi '$ is also given by
\[
\Phi ' (v\otimes (u\otimes a)\otimes x)=\sum (-1)^{|u||v_{1}|+|u||v_{2}|+|a||v_{2}|}uv_{1}\otimes \Bigl( ev^{*}_{2,\frac{1}{2}}\rho (av_{2})\cdot x \Bigr)
\]
for $v\otimes (u\otimes a)\otimes x \in {\mathbb B} \otimes _{A^{\otimes 2}}(({\mathbb B}\otimes A) \otimes _{A^{\otimes 3}}A_{{\rm PL}}(LM^{\times 2}))$, where $\Delta ^{!}(v)=\sum v_{1}\otimes v_{2}$. We here show the commutativity of the diagram \eqref{fro4}. The commutativity of the square (I) is trivial. The naturality of the Eilenberg-Moore map and the commutative diagrams \eqref{fro1}, \eqref{fro2} and \eqref{fro3} show that the diagrams (I\hspace{-.1em}V), (I\hspace{-.1em}I) and (V) commute. The straightforward calculation show the commutativity of (I\hspace{-.1em}I\hspace{-.1em}I), (V\hspace{-.1em}I), (V\hspace{-.1em}I\hspace{-.1em}I\hspace{-.1em}I) and (I\hspace{-.1em}X).
Since $\xi$ is homotopic to the identity map with the homotopy $H' : LM \times I \to LM $ given by
\begin{eqnarray*}
H' (\gamma ,s)(t) =
\left\{
\begin{array}{ll}
\gamma (\frac{2}{2-s}t + \frac{4-3s}{4-2s}) & (0\leq t\leq \frac{1}{4}s) \\
\gamma (t-\frac{1}{4}s) & (\frac{1}{4}s \leq t\leq \frac{1}{4}+\frac{1}{4}s)\\
\gamma (\frac{1}{1+s}t ) & (\frac{1}{4}+\frac{1}{4}s \leq t\leq \frac{1}{2} + \frac{1}{2}s)\\
\gamma (\frac{2}{2-s}t - \frac{3s}{4-2s}) & (\frac{1}{2} + \frac{1}{2}s \leq t\leq 1),
\end{array}
\right.
\end{eqnarray*}
the following shows the commutativity of  (V\hspace{-.1em}I\hspace{-.1em}I);
\begin{align*}
&H({\rm Comp} )H(\Delta ^{!}\otimes 1)(u\otimes x)\\
=&H(\xi {\rm Comp} \circ T)H(\Delta ^{!}\otimes 1)(u\otimes x)\\
=&H(T) H(\xi {\rm Comp} )(j^{*}(\rho \otimes \rho ) \Delta ^{!}(u)\cdot x)\\
=&H(T)\Bigl(   (\xi {\rm Comp} )^{*}j^{*}(\rho \otimes \rho ) \Delta ^{!}(u) \cdot (\xi {\rm Comp} )^{*}(x) \Bigr)\\
=& H(T)\Bigl(  (ev_{1,0} , ev _{1,\frac{1}{2}})^{*}(\rho \otimes \rho ) \Delta ^{!}(u)  \cdot (\xi {\rm Comp} )^{*}(x) \Bigr) \\
=& H(T)H(\Delta ^{!}\otimes 1)(u\otimes (\xi {\rm Comp} )^{*}(x) )\\
=& H(T)H(\Delta ^{!}\otimes 1)( 1\otimes (\xi {\rm Comp})^{*} )(u\otimes x).
\end{align*}
for $u\otimes x$ in $H^{*}({\mathbb B}\otimes _{A^{\otimes 2}}A_{{\rm PL}}(LM))$.
We conclude that
\begin{align*}
 &Dlp\circ Dlcop\\
=& H(\Delta ^{!}\otimes 1)\Phi '(1\otimes {\rm EM}''_{2})^{-1}{\rm EM}''_{1}H(T_{(123)})H({\rm inc} )H({\rm Comp} \times 1)\\
=&(-1)^{d} H((1\otimes \Delta ^{!})\otimes 1)H(\Delta ^{!} \otimes (1\otimes 1)\otimes 1) (1\otimes ({\rm EM}'_{5})^{-1}){\rm EM}_{5}^{-1}H({\rm inc})H({\rm Comp} \times 1)\\
=& (-1)^{d}(1\otimes Dlcop )(Dlp \otimes 1).
\end{align*}
A similar argument shows that the second equation $Dlp \circ Dlcop = (-1)^{d}(Dlcop \otimes 1)(1\otimes Dlp)$. Consider the following three cubes:
\begin{equation}\label{fro5}
\xymatrixrowsep{0.5cm}
\xymatrixcolsep{0.9cm}
\xymatrix{
& LM \times _{M} LM \times _{M} LM \ar[rr]^{1\times {\rm Comp} \circ T } \ar[dd]^(0.3){(1\times {\rm Comp} \circ T) T_{(321)}} \ar[ld]_{ev_{0}}  & & LM \times _{M} LM \ar[dd]^{\xi ' {\rm Comp}}  \ar[ld]^{(ev_{2,\frac{1}{2}},ev_{2,0})} \\
M \ar[dd]^{=} \ar[rr]^(0.3){\Delta} & & M^{\times 2} \ar[dd]^(0.3){=} & \\
& LM \times _{M} LM \ar[ld]^{ev_{0}}  \ar[rr]^{{\rm Comp} } & & LM \ar[ld]^{j} \\
M \ar[rr]^{\Delta }& & M^{\times 2} &
}
\end{equation}

\begin{equation}\label{fro6}
\xymatrixrowsep{0.5cm}
\xymatrixcolsep{0.9cm}
\xymatrix{
& LM \times _{M} LM \times _{M} LM \ar[rr]^{{\rm inc}} \ar[dd]^(0.3){{\rm inc}} \ar[ld]_{ev_{0}}  & &  (LM \times _{M} LM ) \times LM  \ar[dd]^{{\rm inc}}  \ar[ld]^{ev_{0}\times ev_{0}} \\
M \ar[dd]^{\Delta } \ar[rr]^(0.3){\Delta} & & M^{\times 2} \ar[dd]^(0.3){\Delta \times 1} & \\
& LM \times ( LM \times _{M} LM) \ar[ld]^{ev_{0}\times ev_{0}}  \ar[rr]^(0.3){{\rm inc} } & & LM^{\times 3}\ar[ld]^{ev_{0}^{\times 3}} \\
M \times M \ar[rr]^{1\times \Delta }& & M^{\times 3} &
}
\end{equation}

\begin{equation}\label{fro7}
\xymatrixrowsep{0.5cm}
\xymatrixcolsep{0.3cm}
\xymatrix{
& LM \times _{M} LM \times _{M} LM \ar[rr]^-{{\rm inc}} \ar[ld]_{ev_{0}} \ar[dd]^(0.3){T_{(123)}} & &  (LM\times _{M} LM)\times LM  \ar[rr]^-{{\rm Comp} \times 1} \ar[ld]^{ev_{0}\times ev_{0}} & & LM^{\times 2}  \ar[ld]^{j\times 1} \ar[dd]^{R_{\frac{1}{2}}\times 1} \\
M \ar[rr]^(0.3){\Delta } \ar[dd]^{=} & & M ^{\times 2} \ar[rr]^{\Delta \times 1} & & M^{\times 3} \ar[dd]^{=} &\\
& LM \times _{M} LM \times _{M} LM \ar[rr]^{1\times {\rm Comp} \circ T}  \ar[ld]^{ev_{0}} & & LM \times _{M} LM \ar[rr]^{{\rm inc} \circ  T} \ar[ld]_(0.7){(ev_{2,\frac{1}{2}},ev_{2,0})}  & & LM^{\times 2} \ar[ld]^{(ev_{\frac{1}{2}} , ev_{0} ) \times ev_{0}} \\
M \ar[rr]^{\Delta } & & M ^{\times 2} \ar[rr]^{1\times \Delta } & & M^{\times 3} &
}
\end{equation}
The map $\xi ':LM\to LM$ is given by
\begin{eqnarray*}
\xi' (\gamma )(t) =
\left\{
\begin{array}{ll}
\gamma (\frac{1}{2}t + \frac{3}{4}) & (0\leq t\leq \frac{1}{2}) \\
\gamma (2t-1) & (\frac{1}{2}\leq t\leq \frac{3}{4})\\
\gamma (t-\frac{1}{4}) & (\frac{3}{4}\leq t\leq 1).
\end{array}
\right.
\end{eqnarray*}
We denote by $R_{\frac{1}{2}}$ the rotation of loops by $\frac{1}{2}$, that is,
\begin{eqnarray*}
R_{\frac{1}{2}} (\gamma )(t) =
\left\{
\begin{array}{ll}
\gamma (t+\frac{1}{2}) & (0\leq t\leq \frac{1}{2}) \\
\gamma (t-\frac{1}{2}) & (\frac{1}{2}\leq t\leq 1).
\end{array}
\right.
\end{eqnarray*}

Denote by
\begin{align*}
& {\rm EM ''}_{3}: {\mathbb B} \otimes _{A^{\otimes 2}}A_{{\rm PL}}(LM \times _{M}LM) \longrightarrow A_{{\rm PL}}(LM \times _{M}LM \times _{M} LM),\\
& {\rm EM ''}_{4}: (A\otimes {\mathbb B}) \otimes _{A^{\otimes 3}}A_{{\rm PL}}(LM \times LM) \longrightarrow A_{{\rm PL}}(LM \times _{M}LM )
\end{align*}
the Eilenberg-Moore maps of the top square of the diagram \eqref{fro5} and the bottom and right-hand side square of \eqref{fro7}, respectively. Consider the following diagram:
\begin{landscape}
\begin{equation}\label{fro8}
\xymatrixrowsep{0.8cm}
\xymatrixcolsep{0.5cm}
\xymatrix{
H^{*}(LM \times LM) \ar[r]_{H(1\times {\rm Comp} )} \ar[dd]^{H({\rm inc})}
\ar@/^6mm/[rrr]^{(1\otimes Dlp)} \ar@{}[rdd]|{{\rm (I')}} & H^{*}(LM\times (LM ^{\times _{M}2})) \ar[dd]^{H({\rm inc})} \ar@{}[rdd]|{{\rm (I\hspace{-.1em}I')}} \ar[r]_{({\rm EM}_{1})^{-1}}^{\cong }& H^{*}((A\otimes {\mathbb B})\otimes _{A^{\otimes 3}}A_{{\rm PL}}(LM^{\times 3}))  \ar[dd]^{\Delta ^{*}\otimes _{(\Delta \times 1)^{*}}{\rm inc}^{*}} \ar[r]_{H( (1\otimes \Delta ^{!})\otimes 1)} \ar@{}[rdd]|{{\rm (I\hspace{-.1em}I\hspace{-.1em}I')}}& H^{*}(LM^{\times 3}) \ar[dd]_{H({\rm inc} \times 1)}
\\
&&&\\
H^{*}(LM \times _{M} LM) \ar[r]^{H(1\times _{M} {\rm Comp} \circ T)}  \ar@{}[rdddd]|{{\rm (I\hspace{-.1em}V')}} \ar[dddd]^{({\rm EM'})^{-1}}_{\cong }& H^{*}(LM ^{\times _{M}3}) \ar@{}[rdddddd]|(0.35){{\rm (V')}} \ar[r]^-{({\rm EM}_{5})^{-1}}_-{\cong } & H^{*}({\mathbb B}\otimes _{A^{\otimes 2}}A_{{\rm PL}}((LM ^{\times _{M}2})\times LM ))   \ar[r]^{H(\Delta ^{!} \otimes 1)} \ar@{}[rdd]|{{\rm (V\hspace{-.1em}I')}} \ar[dd]^{(1\otimes {\rm EM'}_{1})^{-1}}_{\cong }& H^{*}( (LM ^{\times _{M}2})\times LM) \ar[dd]_{({\rm EM}'_{1})^{-1}}^{\cong }\\
&&&\\
 & H^{*}(LM ^{\times _{M}3}) \ar[uu]^-{\cong }_-{H(T_{(123)})=H(T_{(321)})^{-1}} \ar[dd]^{({\rm EM}''_{3})^{-1}}_{\cong } & H^{*}({\mathbb B}\otimes _{A^{\otimes 2}}(({\mathbb B}\otimes A) \otimes _{A^{\otimes 3}}A_{{\rm PL}}(LM ^{\times 2})))  \ar[r]^-{H(\Delta ^{!} \otimes (1\otimes 1)\otimes 1)}  \ar@{}[rdddd]^{{\rm (I\hspace{-.1em}X')}}& H^{*}(( {\mathbb B}\otimes A)\otimes _{A^{\otimes 3}}A_{{\rm PL}}(LM ^{\times 2}))   \ar[dddd]_{H( (\Delta ^{!}\otimes 1)\otimes 1)} \\
 &&&\\
H^{*}({\mathbb B}\otimes _{A^{\otimes 2}}A_{{\rm PL}}(LM))  \ar[r]^{1\otimes (\xi '  {\rm Comp})^{*} } \ar[dd]^{H(\Delta ^{!} \otimes 1)}  \ar@{}[rdd]|{{\rm (V\hspace{-.1em}I\hspace{-.1em}I')}}  & H^{*}({\mathbb B} \otimes _{A^{\otimes 2}}A_{{\rm PL}}(LM ^{\times _{M}2}))  \ar[dd]^{H(T) H(\Delta ^{!} \otimes 1)}  \ar@{}[rdd]|{{\rm (V\hspace{-.1em}I\hspace{-.1em}I\hspace{-.1em}I')}} \ar[r]^-{(1\otimes {\rm EM}''_{4})^{-1}}_-{\cong } & H^{*}({\mathbb B} \otimes _{A^{\otimes 2}}((A\otimes {\mathbb B}) \otimes _{A^{\otimes 3}}A_{{\rm PL}}(LM^{\times 2})))   \ar[uu]^{\cong }_{\Psi   }  \ar[dd]^{\Psi  '}  & \\
&&&\\
H^{*}(LM) \ar[r]^{H({\rm Comp})} \ar@/_6mm/[rrr]^{Dlp} & H^{*}(LM ^{\times _{M}2}) \ar[r]^{({\rm EM})^{-1}}_{\cong }& H^{*}({\mathbb B} \otimes _{A^{\otimes 2}}A_{{\rm PL}}(LM^{\times 2}))  \ar[r]^{(-1)^{d}H(\Delta ^{!}\otimes 1)} & H^{*}(LM\times LM).
}
\end{equation}
\end{landscape}
We defined the map $\Psi $ and $\Psi '$ for
\begin{align*}
&\Psi (v\otimes (a\otimes u)\otimes x) =  (-1)^{|u|(|a|+|v|)+|a||v|}u\otimes (av\otimes 1)\otimes (R_{\frac{1}{2}}\times 1)^{*}(x),\\
& \Psi ' (v\otimes (a\otimes u)\otimes x) = \sum (-1)^{|v_{2}||a|+(|v_{1}|+|a|)(|v_{2}|+|u|)} v_{2}u \otimes \Bigl( ev_{1,\frac{1}{2}}^{*}\rho (v_{1}a)\cdot x \Bigr), 
\end{align*}
respectively. We also see that $\Psi $ is an isomorphism with the inverse map $\Psi ^{-1}$ given by
\[
\Psi ^{-1}(v \otimes (u\otimes a)\otimes x)=(-1)^{|v|(|a|+|u|)}u\otimes (1\otimes av)\otimes (R_{\frac{1}{2}}\times 1)^{*}(x).
\]
Similar for the argument of the proof the commutativity of the diagram \eqref{fro4}, we see that the diagram \eqref{fro8} commutes. We only show the commutativity for the diagrams (I') and (I\hspace{-.1em}X'). Since the composite ${\rm Comp} \circ T$ is homotopic to ${\rm Comp}$, the diagrams (I') commutes.
It is readily seen that the equality
\[
H(R_{\frac{1}{2}}\times 1)H((\Delta ^{!}\otimes 1)\otimes 1)H(\Delta ^{!} \otimes (1\otimes 1)\otimes 1)\Psi =(-1)^{d}H( \Delta ^{!}\otimes 1)\Psi '
\]
holds by a straightforward calculation.
Since the map $R_{\frac{1}{2}}$ is homotopic to the identity map, we therefore conclude that
\[
Dlp \circ Dlcop = (-1)^{d}(Dlcop \otimes 1)(1\otimes Dlp).
\]
\end{proof}


\section{A triviality or non-triviality of the loop product and the loop coproduct.}

We first introduce a semifree resolution of a minimal Sullivan model $(\Lambda V ,d)$ for $M$ as a $\Lambda V \otimes \Lambda V$-module. Consider the commutative graded algebra $\Lambda V \otimes \Lambda V \otimes \Lambda (sV)$ with the differential $D$ given by
\begin{align*}
&D(v\otimes 1\otimes 1)= d(v)\otimes 1\otimes 1, D(1\otimes v \otimes 1)= 1\otimes d(v) \otimes 1,\\
&D(1\otimes 1\otimes sv) =(-v\otimes 1 + 1\otimes v)\otimes 1 - \sum _{i=1}^{\infty }\frac{(sD)^{i}}{i!}(v\otimes 1\otimes 1),
\end{align*}
where $sV$ is the suspension of $V$, that is $(sV)^{n}=V^{n+1}$, and $s$ is the unique derivation of the algebra $\Lambda V \otimes \Lambda V \otimes \Lambda (sV)$ defined by
\[
s(v\otimes 1\otimes 1) = 1\otimes 1\otimes sv = s(1\otimes v\otimes 1), \ s(1\otimes 1\otimes sv)=0.
\]
Then, by \cite[\S 15, Example1]{Rat}, $(\Lambda V \otimes \Lambda V \otimes \Lambda (sV) ,D)$ is a Sullivan model for $M^{I}$ and
\[
\bar{\varepsilon }:= \mu \cdot \varepsilon  : \Lambda V \otimes \Lambda V \otimes \Lambda (sV) \to \Lambda V
\]
is a semifree resolution of $\Lambda V$ as a $\Lambda V \otimes \Lambda V$-module. Here, $\mu$ is the product of $\Lambda V$ and $\varepsilon :\Lambda V\to {\mathbb Q}$ is the canonical augmentation. Moreover, the commutative differential graded algebra
\[
(\Lambda V \otimes \Lambda (sV) , \bar{d})\cong (\Lambda V , d)\otimes _{\Lambda V\otimes \Lambda V}(\Lambda V \otimes \Lambda V \otimes \Lambda (sV) ,D)
\]
is a Sullivan model for the free loop space $LM$, where the differential $\bar{d}$ is defined as $\bar{d}(v)=d(v)$, $d(sv)=-sd(v)$. 
For simplicity, we put
\[
{\mathcal M}_{M^{I}}=(\Lambda V \otimes \Lambda V \otimes \Lambda (sV) ,D), \ \ {\mathcal M}_{LM}= (\Lambda V \otimes \Lambda (sV), \bar{d}).
\]
We next consider a model for the dual loop product and coproduct. The following is a rational coefficient version of the torsion functor description of \cite{KMN}. If $M$ is a Poincar\'{e} duality space, models for the products are introduced by \cite{CT2007}, \cite{FT2009} and \cite{FTV2007}.
By the following commutative diagram,
\[
\xymatrixrowsep{0.3cm}
\xymatrixcolsep{1cm}
\xymatrix{
       & LM \times _{M} LM \ar[rr]^{pr_{2}} \ar[ld]_{pr_{1}} \ar[dd]^(0.3){{\rm Comp}} &      & LM \ar[ld]^(0.3){ev_{0}} \ar[dd]^{{\rm inc}} \\
LM  \ar[rr]^(0.3){ev_{0}} \ar[dd]^{{\rm inc}}  &                          &  M  \ar[dd]^(0.3){=} &    \\
       & LM    \ar[rr]^{\zeta   _{1}}  \ar[ld]_{\zeta  _{2}}                &       &   M^{I} \ar[ld]^-{(ev_{0},ev_{1})} \\
M^{I} \ar[rr]^(0.7){(ev_{1},ev_{0})} &                         & M\times M & \\
       & LM   \ar[uu]^(0.7){\theta   } \ar[rr]^(0.3){{\rm inc}}   \ar[ld]_{ev_{0}}               &       &   M^{I} \ar[uu]^(0.3){=} \ar[ld]^-{(ev_{0},ev_{1})} \\
 M  \ar[uu]^{c} \ar[rr]^{\Delta }  &                        & M\times M \ar[uu]^(0.3){=}&
}
\]
the morphism
\[
\xymatrixcolsep{0.5cm}
\xymatrix{
{\mathcal M}_{LM}\cong \Lambda V \otimes _{\Lambda V ^{\otimes 2}} {\mathcal M}_{M^{I}} 
& {\mathcal M}_{M^{I}} \otimes _{\Lambda V ^{\otimes 2}}{\mathcal M}_{M^{I}}  \ar[l]_-{\bar{\varepsilon } \otimes 1}^-{\simeq } \ar[d]^-{(\mu\otimes 1)\otimes _{\mu}(\mu \otimes 1)}\\
 \Lambda V \otimes _{\Lambda V^{\otimes 2}}{\mathcal M}_{LM}^{\otimes 2} & {\mathcal M}_{LM}\otimes _{\Lambda V}{\mathcal M}_{LM} \ar[l]^{\cong }
}
\]
induces the map $H({\rm Comp} )=H((\mu\otimes 1)\otimes _{\mu}(\mu \otimes 1))H(\bar{\varepsilon } \otimes 1)^{-1}$ in homology. Here, $c:M\to LM$ is the map which sends $x$ to the constant path at $x$. The maps $\theta  $, $\zeta _{1}$ and $\zeta _{2}$ are given as follows:
\begin{eqnarray*}
\theta  (\gamma )(t) =
\left\{
\begin{array}{ll}
\gamma (2t) & (0\leq t\leq \frac{1}{2} ) \\
\gamma (1) & l( \frac{1}{2}\leq t\leq 1)
\end{array}
\right.
\end{eqnarray*}
and $\zeta _{1}(\gamma )(t)=\gamma (\frac{1}{2}t )$, $\zeta _{2}( \gamma )(t)=\gamma (\frac{1}{2}t+\frac{1}{2})$ for $\gamma \in LM$.
Hence, the dual loop product is induced by the following composite in homology
\begin{equation}\label{mdlp}
\xymatrix{
{\mathcal M}_{LM} \ar[r]^-{\cong} & \Lambda V \otimes _{\Lambda V ^{\otimes 2}} {\mathcal M}_{M^{I}} 
& {\mathcal M}_{M^{I}} \otimes _{\Lambda V ^{\otimes 2}}{\mathcal M}_{M^{I}}  \ar[l]_-{\bar{\varepsilon } \otimes 1}^-{\simeq } \ar[d]^{(\mu\otimes 1)\otimes _{\mu}(\mu \otimes 1)} \\
&& {\mathcal M}_{LM}\otimes _{\Lambda V}{\mathcal M}_{LM} \ar[d]^{\cong } \\
{\mathcal M}_{LM}^{\otimes 2}  & {\mathcal M}_{M^{I}}\otimes _{\Lambda V^{\otimes 2}}{\mathcal M}_{LM}^{\otimes 2}  \ar[r]^-{\bar{\varepsilon } \otimes 1}_-{\simeq }   \ar[l]_-{\Delta^{!}\otimes 1\otimes 1}&  \Lambda V \otimes _{\Lambda V^{\otimes 2}}{\mathcal M}_{LM}^{\otimes 2}.
}
\end{equation}
Consider the two commutative diagrams,
\begin{equation}\label{model1}
\xymatrixrowsep{0.3cm}
\xymatrixcolsep{1cm}
\xymatrix{
                       &  M \times M  \ar[dd]^{=} &    \\
       & LM    \ar[rr]^{\zeta  _{1}}  \ar[ld]_{\zeta _{2}} \ar[dd]^(0.3){j}                &       &   M^{I} \ar[ld]^{(ev_{0},ev_{1})} \ar[dd]^{(ev_{0},ev_{1})} \\
M^{I} \ar[rr]^(0.7){(ev_{1},ev_{0})} \ar[dd]^{(ev_{1},ev_{0})} &                         & M\times M \ar[dd]^(0.3){=} & \\
       & M\times M   \ar[ld]^{=} \ar[rr]^(0.3){=}                  &       &   M\times M \ar[ld]^{=} \\
 M\times M   \ar[rr]^{=} &                        & M\times M &
}
\end{equation}
\begin{equation}\label{model2}
\xymatrixrowsep{0.3cm}
\xymatrixcolsep{1cm}
\xymatrix{
       & LM \times _{M} LM \ar[rr]^{{\rm Comp} } \ar[ld]_{ev_{0}} \ar[dd]^(0.3){{\rm inc} } &      & LM \ar[ld]_{j} \ar[dd]^{(\zeta _{1},\zeta _{2})} \\ 
M  \ar[rr]^(0.3){\Delta } \ar[dd]^{\Delta }  &                          &  M \times M  \ar[dd]^(0.3){ \Delta '  } &    \\
       & LM \times LM   \ar[rr]^(0.3){{\rm inc} }  \ar[ld]_{(ev_{0},ev_{0})}                 &       &   M^{I}\times M^{I} \ar[ld]^-{(ev_{0},ev_{1})^{\times 2}}  \\
M\times M \ar[rr]^{\Delta \times \Delta }  &                         & M^{\times 2} \times M^{\times 2}, & 
}
\end{equation}
where $\Delta '$ is a map which sends $(x,y)$ to $(x,y,y,x)$. By the diagram \eqref{model1}, we see that the inclusion $\Lambda V\otimes \Lambda V \hookrightarrow  {\mathcal M}_{M^{I}}\otimes _{\Lambda V^{\otimes 2}}{\mathcal M}_{M^{I}}$
is a model for $j$. Since the quotient map
\[
\xymatrix{
\overline{\zeta  }:{\mathcal M}_{M^{I}}\otimes {\mathcal M}_{M^{I}} \ar[r] & {\mathcal M}_{M^{I}}\otimes _{\Lambda V^{\otimes 2}} {\mathcal M}_{M^{I}} 
}
\]
is a model for $(\zeta _{1},\zeta _{2})$, by the diagram \eqref{model2}, the following composite is a model for ${\rm inc} : LM\times _{M}LM \to LM \times LM$;
\[
\xymatrix{
{\mathcal M}_{LM}^{\otimes 2}\cong \Lambda V^{\otimes 2}\otimes _{\Lambda V^{\otimes 4}}{\mathcal M}_{M^{I}}^{\otimes 2} \ar[r]^-{\mu \otimes _{\mu '} \overline{\zeta  }} & \Lambda V\otimes _{\Lambda V^{\otimes 2}}({\mathcal M}_{M^{I}}\otimes _{\Lambda V^{\otimes 2}} {\mathcal M}_{M^{I}}) ,
}
\]
where $\mu'$ is a model for $\Delta '$, that is, $\mu'(v_{1}\otimes  v_{2}\otimes v_{3}\otimes v_{4})=(-1)^{|v_{4}|(|v_{2}|+|v_{3}|)}v_{1}v_{4}\otimes v_{2}v_{3}$. We thus see that the dual loop coproduct is induced by the composite in homology;
\begin{equation}\label{mdlcp}
\xymatrix{
{\mathcal M}_{LM}^{\otimes 2}\cong \Lambda V^{\otimes 2}\otimes _{\Lambda V^{\otimes 4}}{\mathcal M}_{M^{I}}^{\otimes 2} \ar[r]^-{\mu \otimes _{\mu '} \overline{\zeta  }}
& \Lambda V\otimes _{\Lambda V^{\otimes 2}}({\mathcal M}_{M^{I}}\otimes _{\Lambda V^{\otimes 2}} {\mathcal M}_{M^{I}}) \\
& {\mathcal M}_{M^{I}}\otimes _{\Lambda V^{\otimes 2}}({\mathcal M}_{M^{I}}\otimes _{\Lambda V^{\otimes 2}} {\mathcal M}_{M^{I}}) \ar[u]_-{\bar{\varepsilon } \otimes 1}^-{\simeq } \ar[d]^-{\Delta ^{!}\otimes 1} \\
{\mathcal M}_{LM} & {\mathcal M}_{M^{I}}\otimes _{\Lambda V^{\otimes 2}} {\mathcal M}_{M^{I}} \ar[l]_-{\bar{\varepsilon } \otimes 1}^-{\simeq }
}
\end{equation}

The following is the result of F\'{e}lix, Halperin and Thomas related to rational Gorenstein spaces.

\begin{thm}\cite[Proposition 3.4, Proposition 5.1]{FHT1988} \label{thm6.1}
Let $X$ be simply-connected space and assume that the rational homotopy group $\pi_{*}(X)\otimes {\mathbb Q}$ is finite dimension. Then, $X$ is a ${\mathbb Q}$-Gorenstein space with formal dimension
\[
\sum _{|x_{i}|: {\rm odd}} |x_{i}| - \sum _{|x_{i}|:{\rm even}} (|x_{i}|-1),
\]
where $x_{i}$ is a basis of $\pi_{*}(X)\otimes {\mathbb Q}$.
\end{thm}

Let $(\Lambda V,d)$ be a minimal Sullivan model for a simply-connected space $M$. Since $V$ is isomorphic to ${\rm Hom}_{{\mathbb Z}}(\pi_{*}(X),{\mathbb Q})$ (\cite[Lemma 13.11]{Rat}), if $V$ is finite dimensional, then $M$ is a Gorenstein space. We now put
\[
{\rm fdim} \, M = \sum _{|x_{i}|: \text{odd}} |x_{i}| - \sum _{|x_{i}|:\text{even}} (|x_{i}|-1).
\]

Before proving Proposition \ref{trivial1}, we give the following lemma.
\begin{lem}\label{odd}
For any odd degree elements $x_{1},x_{2} \cdots x_{k}$ in $V$,
\[
\Bigl( \prod _{i=1}^{k}(-x_{i}\otimes 1 +1\otimes x_{i}) \Bigr) (-x_{1}x_{2}\cdots x_{k}\otimes 1+1\otimes x_{1}x_{2}\cdots x_{k})=0
\] 
in $\Lambda V\otimes \Lambda V$.
\end{lem}
\begin{proof}
The proof is induction on $k$. For $k=1$, it is easily seen that $(-x_{1}\otimes 1+1\otimes x_{1})^{2}=0$.
Assume that the equation
\[
\Bigl( \prod _{i=i}^{k-1} (-x_{i}\otimes 1+1\otimes x_{i}) \Bigr) (-x_{1}x_{2}\cdots x_{k-1} \otimes 1 + 1\otimes x_{1}x_{2}\cdots x_{k-1} )=0 
\]
hold. Then,
\begin{align*}
 &(-x_{k}\otimes 1 +1\otimes x_{k})(-x_{1}x_{2}\cdots x_{k}\otimes 1+1\otimes x_{1}x_{2}\cdots x_{k})\\
=&(-x_{k}\otimes 1 +1\otimes x_{k})\Bigl( (-x_{1}x_{2}\cdots x_{k-1}\otimes 1 + 1\otimes x_{1}x_{2}\cdots x_{k-1})(x_{k}\otimes 1+1\otimes x_{k}) \\
   & \hspace{10em} -(-1)^{k-1}x_{k}\otimes x_{1}x_{2}\cdots x_{k-1}+x_{1}x_{2}\cdots x_{k-1}\otimes x_{k} \Bigr)\\
=& (-1)^{k}(x_{k}\otimes x_{k})(-x_{1}x_{2}\cdots x_{k-1}\otimes 1+ 1\otimes x_{1}x_{2}\cdots x_{k-1}).
\end{align*}
Hence, induction hypothesis shows the assertion.
\end{proof}
\begin{proof}[Proof of Proposition 1.2]
(1) Let $x_{1},x_{2},\cdots , x_{n}$ be a homogeneous basis of $V=V^{\text{odd}}$. We see that a $\Lambda V^{\otimes 2}$-module map $\Delta ^{!}:{\mathcal M}_{M^{I}}\to \Lambda V^{\otimes 2}$ defied by
\[
\Delta ^{!}(1)=\prod _{i=1}^{n}(-x_{i}\otimes 1+1\otimes x_{i}), \ \Delta ^{!}(sx_{i_{1}}\cdots sx_{i_{k}})=0
\]
is a generator of ${\rm Ext}^{{\rm fdim}\, M}_{\Lambda V^{\otimes 2}}(\Lambda V,\Lambda V^{\otimes 2})$. We now check that the map $\Delta ^{!}$ is a cycle but it is not boundary. For any $x_{i}$, we may write $dx_{i}=\sum \lambda x_{i_{1}}\cdots x_{i_{k}}$ for some $\lambda \in {\mathbb Q}$. Then, by Lemma \ref{odd},
\begin{align*}
d\Delta ^{!}(1) = d\Bigl( \prod _{i=1}^{n}(-x_{i}\otimes 1+1\otimes x_{i}) \Bigr) =0
\end{align*}
Hence, $(d\Delta ^{!}-(-1)^{d}\Delta ^{!}D)(1)=0$.
The equality $(-x_{i}\otimes 1+1\otimes x_{i})^{2}=0$ enables us to obtain that $(d\Delta ^{!}-(-1)^{d}\Delta ^{!}D)(sx_{i})=0$, and similarly, we have $(d\Delta ^{!}-(-1)^{d}\Delta ^{!}D)(sx_{i_{1}}\cdots sx_{i_{k}})=0$. Hence $\Delta ^{!}$ is a cycle in ${\rm Hom}^{*}_{\Delta V^{\otimes 2}}({\mathcal M}_{M^{I}},\Lambda V^{\otimes 2})$. Any $\Lambda V^{\otimes 2}$-module map ${\mathcal M}_{M^{I}}\to \Lambda V^{\otimes 2}$ with degree ${\rm fdim} \, M -1$ send $1$ to $0$ by degree reason, the map $\Delta ^{!}$ is not a boundary. Since $\Delta ^{!}(1)$ is a non-trivial cycle, the dual loop product is non-trivial. The equality $\mu \Delta ^{!}=0$ implies that the dual loop coproduct is trivial.\\
\indent
(2) We first note that if $V$ is generated by even degree basis, the differential $d$ is zero by degree reason. Let $y_{1},y_{2}, \cdots ,y_{m}$ be a basis of $V=V^{\text{even}}$. Then, the $\Lambda V^{\otimes 2}$-module map $\Delta^{!}:{\mathcal M}_{M^{I}}\to \Lambda V^{\otimes 2}$ defined by $\Delta^{!}(1)=0$ and
\begin{eqnarray*}
\Delta^{!} (sy_{j_{1}}sy_{j_{2}}\cdots sy_{j_{l}}) =
\left\{
\begin{array}{cl}
1 & (\{ j_{1},j_{2},\cdots ,j_{k} \} = \{ 1,2,\cdots ,m\} )\\
0 & (\text{otherwise}).
\end{array}
\right.
\end{eqnarray*}
is a generator of ${\rm Ext}^{{\rm fdim}\, M}_{\Lambda V^{\otimes 2}}(\Lambda V,\Lambda V^{\otimes 2})$. Indeed, a straightforward calculation shows that the map $\Delta^{!}$ is compatible with the differentials. Also, for any $\Lambda V^{\otimes 2}$-module map $\psi :{\mathcal M}_{M^{I}}\to \Lambda V^{\otimes 2}$ of degree ${\rm fdim}\, M -1$, the following equation shows that $\Delta^{!}$ is not a boundary:
\begin{align*}
   &(-1)^{{\rm fdim}\, M -1}\psi D ((sy_{1}\cdots sy_{m}) ) \\
=&(-1)^{{\rm fdim}\, M -1} \sum _{i=1}^{m}\pm \psi ( (-y_{i}\otimes 1 +1\otimes y_{i} ) \otimes sy_{1}\cdots sy_{i-1}sy_{i+1}\cdots sy_{m}) \\
= &(-1)^{{\rm fdim}\, M -1}\sum _{i=1}^{m}\pm (-y_{i}\otimes 1 +1\otimes y_{i})\psi (sy_{1}\cdots sy_{i-1}sy_{i+1}\cdots sy_{m}) \neq 1,
\end{align*}
where $\pm$ is the Kuszul sign convention. By \eqref{mdlp}, we have that
\begin{align*}
&H(\bar{\varepsilon }\otimes 1)^{-1}H({\rm Comp} )(v_{1}v_{2}\cdots v_{l}\otimes sw_{1}sw_{2}\cdots sw_{k}) \\
=& \prod _{i=1}^{k}(v_{1}\cdots v_{l}\otimes 1 \otimes 1)\otimes \Bigl( (1\otimes 1)\otimes (1\otimes sw_{i}) - (1\otimes sw_{i})\otimes (1\otimes 1) \Bigr).
\end{align*}
for any $v_{1}v_{2}\cdots v_{l}\otimes sw_{1}\cdots sw_{k}$ in ${\mathcal M}_{LM}$. Therefore, $\Delta^{!}(1)=0$ implies that the dual loop product is trivial. Also the following equations show that the dual loop coproduct is non-trivial:
\begin{align*}
&Dlcop ( (1\otimes sy_{1}sy_{2}\cdots sy_{m}) \otimes (1\otimes 1))\\
=& H(\bar{\varepsilon }\otimes 1)H(\Delta ^{!}\otimes 1)\Bigl( \ \prod_{i=1}^{m} ( -(1\otimes 1\otimes sy_{i})\otimes (1\otimes 1\otimes 1)\otimes (1\otimes 1 \otimes 1  ) \\
 & \hspace{11em}  (1\otimes 1\otimes 1)\otimes (1\otimes 1\otimes sy_{i})\otimes (1\otimes 1 \otimes 1  )\Bigr)\\
 =& H(\bar{\varepsilon }\otimes 1)( (-1)^{m}(1\otimes 1\otimes 1)\otimes (1\otimes 1\otimes 1)  )= (-1)^{m}(1\otimes 1) \neq 0.
\end{align*}
This completes the proof.
\end{proof}

We next consider spaces in which a minimal Sullivan model of the spaces are pure and first recall the definition of pure Sullivan algebras.
\begin{defn}\cite[\S 32 (a)]{Rat} \label{pure}
{\rm
A minimal Sullivan model is {\it pure} if $V$ is finite dimensional, $d(V^{\text{odd}})\subset \Lambda V^{\text{even}}$ and $d(V^{\text{even}})=0$.
}
\end{defn}

Let $(\Lambda V,d)$ be a pure minimal Sullivan model with $V^{\text{odd}}\neq \{0 \}$, $V^{\text{even}}\neq \{ 0 \}$ and $x_{1},\cdots ,x_{n}$ is a basis of $V^{\text{odd}}$ and $y_{1},\cdots ,y_{m}$ is a basis of $V^{\text{even}}$. Then, we may write
\[
D(sx_{r}) = (-x_{r}\otimes 1 + 1\otimes x_{r})\otimes 1 -\sum _{i=1}^{m}f_{i}^{r}\otimes sy_{i}
\]
in ${\mathcal M}_{M^{I}}$ for some $f_{i}^{r}\in \Lambda V^{\otimes 2}$. For a subset $J=\{ j_{1}<j_{2}<\cdots <j_{k}\}$ of $\{ 1,2,\cdots ,m\}$, we put
\[
sy_{J}=sy_{j_{1}}sy_{j_{2}}\cdots sy_{j_{k}}
\]
for simplicity. Especially, if $J$ is the empty set $\phi $, put $sy_{\phi }=1$. We now define a $\Lambda V^{\otimes 2}$-module map $\Delta^{!}:{\mathcal M}_{M^{I}}\to \Lambda V^{\otimes 2}$ as follows:
for any generator of $\Lambda (sV)$, if a generator is of the form $sy_{J^{c}}$ for some $J\subset \{ 1,2,\cdots ,m\}$,
\[
\Delta^{!}(sy_{J^{c}})=\sum _{i_{1}=1}^{n}\sum _{{\begin{subarray}{c}i_{2}=1,\\ i_{2}\neq i_{1} \end{subarray}}}^{n}\cdots \sum_{\begin{subarray}{c}i_{k}=1,\\ i_{k}\neq i_{1},\cdots ,i_{k-1} \end{subarray}}^{n} \hspace{-1em}(-1)^{\varepsilon _{(J,i_{1},\cdots ,i_{k})}}f_{j_{1}}^{i_{1}}\cdots f_{j_{k}}^{i_{k}} \Bigl( \prod_{\begin{subarray}{c}i=1,\\ i\neq i_{1},\cdots ,i_{k} \end{subarray}}^{n}(-x_{i}\otimes 1+1\otimes x_{i}) \Bigr)
\]
and $\Delta^{!} $ sends the others to $0$. Here, $J^{c}$ is the complementary subset of $J$, 
\[
\varepsilon _{(J,i_{1},\cdots ,i_{k})}=\sum_{r=1}^{k}(i_{r}+j_{r}+r-1)+k{\rm fdim}\, M + s(\sigma ),
\]
$\sigma $ is a $k$-permutation which satisfies $i_{\sigma (k)}<i_{\sigma (k-1)}<\cdots <i_{\sigma (1)}$. If $\sigma $ is a even permutation, we put $s(\sigma )=0$ and if $\sigma $ is a odd permutation, put $s(\sigma )=1$.
\begin{lem}\label{pureshriek}
The map $\Delta^{!}$ is a generator of ${\rm Ext}_{\Lambda V^{\otimes 2}}^{{\rm fdim}\, M}(\Lambda V,\Lambda V^{\otimes 2})$.
\end{lem}
\begin{proof}
It is only enough to check that the map $\Delta^{!}$ is a cycle and not a boundary in ${\rm Hom}_{\Lambda V^{\otimes 2}}^{*}({\mathcal M}_{M^{I}},\Lambda V^{\otimes 2})$.
Let $\{i_{1},i_{2},\cdots ,i_{k}\}$ be a subset of $\{1,2,\cdots ,n\}$ and $\sigma $ is a $k$-permutation such that $i_{\sigma (k)}<i_{\sigma (k-1)}<\cdots <i_{\sigma (1)}$. Then,
\begin{align*}
  &d\Bigl( \prod_{\begin{subarray}{c}i=1,\\ i\neq i_{1},\cdots ,i_{k} \end{subarray}}^{n}(-x_{i}\otimes 1+1\otimes x_{i}) \Bigr) \\
=& \sum _{r=0}^{k}\sum _{p=i_{\sigma (k-r+1)+1}}^{i_{\sigma (k-r)-1}}(-1)^{p-1-r}(-dx_{p}\otimes 1+1\otimes dx_{p})\prod_{\begin{subarray}{c}i=1,\\ i\neq i_{1},\cdots ,i_{k},p \end{subarray}}^{n}(-x_{i}\otimes 1+1\otimes x_{i}),
\end{align*}
where we put $\sigma (k+1)=0$ and $\sigma (0)=n+1$ for convenient. Since $-dx_{p}\otimes 1+1\otimes dx_{p}=\sum _{j=1}^{m}f^{p}_{j}(-y_{j}\otimes 1+1\otimes y_{j})$,
\begin{align*}
d\Delta^{!}(sy_{J^{c}})=&\sum _{j=1}^{m} \sum _{r=0}^{k}\sum _{p=i_{\sigma (k-r+1)+1}}^{i_{\sigma (k-r)-1}}\sum _{i_{1}=1}^{n}\sum _{{\begin{subarray}{c}i_{2}=1,\\ i_{2}\neq i_{1} \end{subarray}}}^{n}\cdots \sum_{\begin{subarray}{c}i_{k}=1,\\ i_{k}\neq i_{1},\cdots ,i_{k-1} \end{subarray}}^{n}(-1)^{\varepsilon _{(J,i_{1},\cdots ,i_{k})}+p-1-r}\\
&\times (-y_{j}\otimes 1+1\otimes y_{j})f_{j}^{p}f_{j_{1}}^{i_{1}}\cdots f_{j_{k}}^{i_{k}}\prod_{\begin{subarray}{c}i=1,\\ i\neq i_{1},\cdots ,i_{k},p \end{subarray}}^{n}(-x_{i}\otimes 1+1\otimes x_{i}).
\end{align*}
A straightforward calculation shows that
\begin{align*}
&\sum _{j\in J} \sum _{r=0}^{k}\sum _{p=i_{\sigma (k-r+1)+1}}^{i_{\sigma (k-r)-1}}\sum _{i_{1}=1}^{n}\sum _{{\begin{subarray}{c}i_{2}=1,\\ i_{2}\neq i_{1} \end{subarray}}}^{n}\cdots \sum_{\begin{subarray}{c}i_{k}=1,\\ i_{k}\neq i_{1},\cdots ,i_{k-1} \end{subarray}}^{n}(-1)^{\varepsilon _{(J,i_{1},\cdots ,i_{k})}+p-1-r}\\
&\times (-y_{j}\otimes 1+1\otimes y_{j})f_{j}^{p}f_{j_{1}}^{i_{1}}\cdots f_{j_{k}}^{i_{k}}\prod_{\begin{subarray}{c}i=1,\\ i\neq i_{1},\cdots ,i_{k},p \end{subarray}}^{n}(-x_{i}\otimes 1+1\otimes x_{i})=0
\end{align*}
and
\begin{align*}
&\sum _{j\in J^{c}} \sum _{r=0}^{k}\sum _{p=i_{\sigma (k-r+1)+1}}^{i_{\sigma (k-r)-1}}\sum _{i_{1}=1}^{n}\sum _{{\begin{subarray}{c}i_{2}=1,\\ i_{2}\neq i_{1} \end{subarray}}}^{n}\cdots \sum_{\begin{subarray}{c}i_{k}=1,\\ i_{k}\neq i_{1},\cdots ,i_{k-1} \end{subarray}}^{n}(-1)^{\varepsilon _{(J,i_{1},\cdots ,i_{k})}+p-1-r}\\
&\times (-y_{j}\otimes 1+1\otimes y_{j})f_{j}^{p}f_{j_{1}}^{i_{1}}\cdots f_{j_{k}}^{i_{k}}\prod_{\begin{subarray}{c}i=1,\\ i\neq i_{1},\cdots ,i_{k},p \end{subarray}}^{n}(-x_{i}\otimes 1+1\otimes x_{i})=(-1)^{{\rm fdim}\, M}\Delta ^{!}D(sy_{J^{c}}).
\end{align*}
Hence, the equation $(d\Delta^{!}-(-1)^{{\rm fdim}\, M}\Delta^{!} D)(sy_{J^{c}})=0$ holds. Similarly, for a base of $\Lambda (sV)$ of the form $sx_{q}sy_{J^{c}}$, we see that $(d\Delta^{!}-(-1)^{{\rm fdim}\, M}\Delta^{!} D)(sx_{q}sy_{J^{c}})=0$. Indeed,
\begin{align*}
 &(d\Delta^{!}-(-1)^{{\rm fdim}\, M}\Delta^{!} D)(sx_{q}sy_{J^{c}})\\
=&(-1)^{{\rm fdim}\, M +1}\Delta ^{!}D(sx_{q}sy_{J^{c}})\\
=& -(-x_{q}\otimes 1+1\otimes x_{q})\Delta ^{!}(sy_{J^{c}}) + (-1)^{{\rm fdim}\, M}\sum _{r=1}^{k}(-1)^{j_{r}-r}f^{q}_{j_{r}}\Delta^{!}(sy_{(J - \{ j_{r} \} )^{c}}).
\end{align*}
Since
\begin{align}\label{term1}
&\sum _{r=1}^{k}(-1)^{{\rm fdim}\, M+j_{r}-r}f^{q}_{j_{r}}\Delta^{!}(sy_{(J - \{ j_{r} \} )^{c}})\\
=& \sum _{r=1}^{k} \sum _{i_{1}=1}^{n}\cdots \sum _{{\begin{subarray}{c}i_{r-1}=1,\\ i_{r-1}\neq i_{1},\cdots ,i_{r-2} \end{subarray}}}^{n}\sum _{{\begin{subarray}{c}i_{r+1}=1,\\ i_{r+1}\neq i_{1},\cdots ,i_{r-1} \end{subarray}}}^{n}\cdots \sum_{\begin{subarray}{c}i_{k}=1,\\ i_{k}\neq i_{1}, \cdots ,i_{k-1} \end{subarray}}^{n} \notag \\
&(-1)^{{\rm fdim}\, M+j_{r}-r+ \varepsilon _{(J-\{ j_{r}\} ,i_{1},\cdots ,i_{k})}}f^{q}_{j_{r}}f_{j_{1}}^{i_{1}}\cdots f_{j_{k}}^{i_{k}} \Bigl( \prod_{\begin{subarray}{c}i=1,\\ i\neq i_{1},\cdots ,i_{k} \end{subarray}}^{n}(-x_{i}\otimes 1+1\otimes x_{i}) \Bigr) , \notag 
\end{align}
we can decompose the right-hand side of \eqref{term1} to the following two terms;
\begin{align}\label{term2}
& \sum _{r=1}^{k} \sum _{{\begin{subarray}{c}i_{1}=1,\\ i_{1}\neq q \end{subarray}}}^{n}\cdots \sum _{{\begin{subarray}{c}i_{r-1}=1,\\ i_{r-1}\neq q, i_{1},\cdots ,i_{r-2} \end{subarray}}}^{n}\sum _{{\begin{subarray}{c}i_{r+1}=1,\\ i_{r+1}\neq q, i_{1},\cdots ,i_{r-1} \end{subarray}}}^{n}\cdots \sum_{\begin{subarray}{c}i_{k}=1,\\ i_{k}\neq q, i_{1}, \cdots ,i_{k-1} \end{subarray}}^{n}  \\
&(-1)^{{\rm fdim}\, M+j_{r}-r+ \varepsilon _{(J-\{ j_{r}\} ,i_{1},\cdots ,i_{k})}}f^{q}_{j_{r}}f_{j_{1}}^{i_{1}}\cdots f_{j_{k}}^{i_{k}} \Bigl( \prod_{\begin{subarray}{c}i=1,\\ i\neq i_{1},\cdots ,i_{k} \end{subarray}}^{n}(-x_{i}\otimes 1+1\otimes x_{i}) \Bigr)  \notag 
\end{align}
and the sum of the other terms. A straightforward calculation shows that the term \eqref{term2} is equal to $(-x_{q}\otimes 1+1\otimes x_{q})\Delta ^{!}(sy_{J^{c}})$ and the others is zero. Hence, $(d\Delta^{!}-(-1)^{{\rm fdim}\, M}\Delta^{!}D)(sx_{q}sy_{J^{c}})=0$. It is readily seen that $(d\Delta^{!}-(-1)^{{\rm fdim}\, M}\Delta^{!}D)(sx_{1}^{m_{1}}\cdots sx_{n}^{m_{n}}sy_{J^{c}})=0$ for some $m_{i}\geq 0$, it turns out that $\Delta^{!}$ is a cycle in ${\rm Hom}_{\Lambda V^{\otimes 2}}^{*}({\mathcal M}_{M^{I}},\Lambda V^{\otimes 2})$. By definition of $\Delta^{!}$,
\[
\Delta^{!}(sy_{1}\cdots sy_{m})=\prod_{i=1}^{n}(-x_{i}\otimes 1+1\otimes x_{i}).
\]
However, for any $\Lambda V^{\otimes 2}$-module map $\psi $ with degree ${\rm fdim}\, M-1$, we see that $(d\psi -(-1)^{{\rm fdim}\, M -1}\psi D)(sy_{1}\cdots sy_{m})$ is not in $\Lambda ^{\geq 1}V^{\text{odd}}$. It implies that $\Delta^{!}$ is not a boundary, and hence we have the assertion.
\end{proof}

\begin{proof}[Proof of Proposition \ref{trivial2}]
(1) If the differential $d$ is zero, $f^{r}_{i}=0$ for any $r$ and $i$. It turns out that, by the formula of Lemma \ref{pureshriek}, $\Delta^{!}(1)=0$ and $\mu \Delta^{!}=0$. Similarly argument of the proof of Proposition \ref{trivial1} (2), the equality $\Delta^{!}(1)=0$ implies that the dual loop product is trivial. We also see that the dual loop coproduct is trivial by the equation $\mu \Delta^{!}=0$.\\
\indent
(2) By Lemma \ref{pureshriek}, if ${\rm dim}\, V^{\text{odd}} > {\rm dim}\, V^{\text{even}}$, then $\mu \Delta^{!}=0$. Therefore, we see that the dual loop coproduct is trivial.
\end{proof}

\begin{ex}\label{compu}
{\rm
Let $M=ES^{1}\times _{S^{1}}{\mathbb C}P^{2}$ be the Borel construction associated to the action
\[
S^{1}\times {\mathbb C}P^{2}\longrightarrow {\mathbb C}P^{2}, \ t\cdot (x, y , z) = (tx,y,z).
\]
We see that the space $M$ is a ${\mathbb Q}$-Gorenstein space of formal dimension $3$ (\cite[Theorem 4.3]{FHT1988}). By \cite[Example 7.41]{Rat2}, a commutative differential graded algebra $(\Lambda V ,d)=(\Lambda (x_{2},u_{2},w_{5}),d)$ with $|x_{2}|=|u_{2}|=2$, $|w_{5}|=5$ and $dx_{2}=du_{2}=0$, $dw_{5}=u_{2}^{3}+x_{2}u_{2}^{2}$ is a minimal pure Sullivan model for $ES^{1}\times _{S^{1}}{\mathbb C}P^{2}$. A straightforward computation shows that
\begin{align*}
&D(sw_{5})=(-w_{5}\otimes 1 + 1\otimes w_{5})\otimes 1 -f\otimes su_{2} - g\otimes sx_{2},\\
& f = u_{2}^{2}\otimes 1 + u_{2}\otimes u_{2} + 1\otimes u_{2}^{2} + \frac{1}{3}u_{2}\otimes x_{2} + \frac{1}{3}x_{2}\otimes u_{2} + \frac{2}{3}u_{2}x_{2}\otimes 1 + \frac{2}{3}\otimes u_{2}x_{2},\\
& g= \frac{1}{3}u_{2}^{2}\otimes 1 + \frac{1}{3}u_{2}\otimes u_{2} + \frac{1}{3}\otimes u_{2}^{2}
\end{align*}
in ${\mathcal M}_{M^{I}}$. Hence, by Lemma \ref{pureshriek}, the $\Lambda V^{\otimes 2}$-module map $\Delta^{!}$ satisfies that
\begin{align*}
\Delta^{!}(sx_{2}su_{2}) = -w_{5}\otimes 1 + 1\otimes w_{5}, \ \ \Delta^{!}(sx_{2}) = f, \ \ \Delta^{!}(su_{2})=-g, \ \ \Delta^{!}(1)=0.
\end{align*}
It is easily seen that the dual loop coproduct is non-trivial. Indeed, $1\otimes su_{2}$ is a non-zero element in $H^{*}({\mathcal M}_{LM})=H^{*}(LM;{\mathbb Q})$, and by \eqref{mdlcp},
\begin{align*}
&Dlcop ((1\otimes su_{2})\otimes (1\otimes 1))\\
=& H(\bar{\varepsilon }\otimes 1)H(\Delta^{!}\otimes 1)H(\bar{\varepsilon }\otimes 1)^{-1}H(\mu \otimes _{\mu '} \bar{\zeta  })((1\otimes su_{2})\otimes (1\otimes 1))\\
=&H(\bar{\varepsilon }\otimes 1)H(\Delta^{!}\otimes 1)H(\bar{\varepsilon }\otimes 1)^{-1}( 1\otimes (1\otimes 1\otimes su_{2}) \otimes (1\otimes 1\otimes 1) )\\
=& H(\bar{\varepsilon }\otimes 1)H(\Delta^{!}\otimes 1)( - (1\otimes 1\otimes su_{2})\otimes (1\otimes 1\otimes 1) \otimes (1\otimes 1\otimes 1) \\
 & \hspace{10em}  + (1\otimes 1\otimes 1) \otimes (1\otimes 1\otimes su_{2}) \otimes (1\otimes 1\otimes 1))\\
=& H(\bar{\varepsilon }\otimes 1)( (g \otimes 1)\otimes (1\otimes 1\otimes 1))\\
=& u_{2}^{2}\otimes 1 \neq 0.
\end{align*}
By \eqref{mdlp}, $Dlp = H(\Delta ^{!}\otimes 1 \otimes 1)H(\bar{\varepsilon }\otimes 1)^{-1}H((\mu \otimes 1)\otimes _{\mu}(\mu \otimes 1))H(\bar{\varepsilon }\otimes 1)^{-1}$. For the non-zero element $1\otimes sx_{2}su_{2}sw_{5}$ in $H^{*}({\mathcal M}_{LM})=H^{*}(LM;{\mathbb Q})$,
\begin{align*}
&Dlp(1\otimes sx_{2}su_{2}sw_{5})\\
=& H(\Delta ^{!}\otimes 1 \otimes 1)H(\bar{\varepsilon }\otimes 1)^{-1}H((\mu \otimes 1)\otimes _{\mu}(\mu \otimes 1))\\
    & \Bigl( 1_{{\mathcal M}_{M^{I}}}\otimes (1\otimes 1 \otimes sx_{2}su_{2}sw_{5}) - (1\otimes 1\otimes sx_{2})\otimes (1\otimes 1\otimes su_{2}sw_{5})\\
    & + (1\otimes 1\otimes su_{2})\otimes (1\otimes 1\otimes sx_{2}sw_{5}) - (1\otimes 1\otimes sw_{5})\otimes (1\otimes 1\otimes sx_{2}su_{2})\\
    & + (1\otimes 1\otimes sx_{2}su_{2})\otimes (1\otimes 1\otimes sw_{5}) - (1\otimes 1\otimes su_{2}sw_{5})\otimes (1\otimes 1\otimes sx_{2})\\
    & + (1\otimes 1\otimes sx_{2}sw_{5})\otimes (1\otimes 1\otimes su_{2}) - (1\otimes 1\otimes sx_{2}su_{2}sw_{5})\otimes 1_{{\mathcal M}_{M^{I}}}
\Bigr)
\end{align*}
\begin{align*}
=& H(\Delta ^{!}\otimes 1 \otimes 1)  \\
  &\Bigl( \ 1_{{\mathcal M}_{M^{I}}}\otimes (1\otimes 1) \otimes (1\otimes sx_{2}su_{2}sw_{5})- 1_{{\mathcal M}_{M^{I}}}\otimes (1\otimes sx_{2})\otimes (1\otimes su_{2}sw_{5})\\
 &  + 1_{{\mathcal M}_{M^{I}}}\otimes (1\otimes su_{2})\otimes (1\otimes sx_{2}sw_{5})- 1_{{\mathcal M}_{M^{I}}}\otimes (1\otimes sw_{5})\otimes (1\otimes sx_{2}su_{2})\\
 & + 1_{{\mathcal M}_{M^{I}}}\otimes (1\otimes sx_{2}su_{2})\otimes (1\otimes sv_{5}) - 1_{{\mathcal M}_{M^{I}}}\otimes (1\otimes su_{2}sw_{5})\otimes (1\otimes sx_{2})\\
 & + 1_{{\mathcal M}_{M^{I}}}\otimes (1\otimes sx_{2}sw_{5})\otimes (1\otimes su_{2}) - 1_{{\mathcal M}_{M^{I}}}\otimes (1\otimes sx_{2}su_{2}sw_{5})\otimes (1\otimes 1)\\
 &  + 3((u_{2}\otimes 1 + 1\otimes u_{2})\otimes su_{2})\otimes (1\otimes su_{2})\otimes (1\otimes sx_{2}su_{2}) \\
 &  + ((u_{2}\otimes 1 + 1\otimes u_{2})\otimes sx_{2})\otimes (1\otimes su_{2})\otimes (1\otimes sx_{2}su_{2})\\
 & + ((x_{2}\otimes 1+1\otimes x_{2})\otimes su_{2})\otimes (1\otimes su_{2})\otimes (1\otimes sx_{2}su_{2})\\
 & + ((u_{2}\otimes 1+1\otimes u_{2})\otimes su_{2})\otimes (1\otimes sx_{2})\otimes (1\otimes sx_{2}su_{2})\\
 & + 3((u_{2}\otimes 1 + 1\otimes u_{2})\otimes su_{2})\otimes (1\otimes sx_{2}su_{2})\otimes (1\otimes su_{2}) \\
 &  + ((u_{2}\otimes 1 + 1\otimes u_{2})\otimes sx_{2})\otimes (1\otimes sx_{2}su_{2})\otimes (1\otimes su_{2})\\
 & + ((x_{2}\otimes 1+1\otimes x_{2})\otimes su_{2})\otimes (1\otimes sx_{2}su_{2})\otimes (1\otimes su_{2})\\
 & + ((u_{2}\otimes 1+1\otimes u_{2})\otimes su_{2})\otimes (1\otimes sx_{2}su_{2})\otimes (1\otimes sx_{2})
\Bigr)
\end{align*}
\begin{align*}
=& \frac{2}{3}(x_{2}u_{2}\otimes su_{2})\otimes (u_{2}\otimes sx_{2}su_{2}) + \frac{2}{3}(u_{2}\otimes su_{2})\otimes (x_{2}u_{2}\otimes sx_{2}su_{2})\\
&+ \frac{1}{3}(x_{2}u_{2}^{2}\otimes su_{2})\otimes (1\otimes sx_{2}su_{2}) + \frac{1}{3}(1\otimes su_{2})\otimes (x_{2}u_{2}^{2} \otimes sx_{2}su_{2})\\
&- \frac{1}{3}(u_{2}^{3}\otimes sx_{2})\otimes (1\otimes sx_{2}su_{2}) - \frac{2}{3}(u_{2}^{2}\otimes sx_{2})\otimes (u_{2}\otimes sx_{2}su_{2})\\
& - \frac{2}{3}(u_{2}\otimes sx_{2})\otimes (u_{2}^{2}\otimes sx_{2}su_{2}) - \frac{1}{3}(1\otimes sx_{2})\otimes (u_{2}^{3}\otimes sx_{2}su_{2})\\
& + \frac{2}{3}(x_{2}u_{2}\otimes sx_{2}su_{2})\otimes (u_{2}\otimes su_{2}) + \frac{2}{3}(u_{2}\otimes sx_{2}su_{2})\otimes (x_{2}u_{2}\otimes su_{2})\\
&+ \frac{1}{3}(x_{2}u_{2}^{2}\otimes sx_{2}su_{2})\otimes (1\otimes su_{2}) + \frac{1}{3}(1\otimes sx_{2}su_{2})\otimes (x_{2}u_{2}^{2} \otimes su_{2})\\
&- \frac{1}{3}(u_{2}^{3}\otimes sx_{2}su_{2})\otimes (1\otimes sx_{2}) - \frac{2}{3}(u_{2}^{2}\otimes sx_{2}su_{2})\otimes (u_{2}\otimes sx_{2})\\
& - \frac{2}{3}(u_{2}\otimes sx_{2}su_{2})\otimes (u_{2}^{2}\otimes sx_{2}) - \frac{1}{3}(1\otimes sx_{2}su_{2})\otimes (u_{2}^{3}\otimes sx_{2}) \neq 0.
\end{align*}
Here, $1_{{\mathcal M}_{M^{I}}}:=1\otimes 1\otimes 1$ in ${\mathcal M}_{M^{I}}$. Therefore, the dual loop product is non-trivial.\\
\indent
The same calculation described above shows that
\[
Dlcop (Dlcop \otimes 1) ((1\otimes su_{2})\otimes (1\otimes 1) \otimes (1\otimes su_{2})) = u_{2}^{4}\otimes 1 \neq 0.
\]
Therefore, it is an example which $(Lcop\otimes 1)Lcop$ is non-trivial.
}
\end{ex}

\section*{Acknowledgments}
The author would like to thank his adviser, Katsuhiko Kuribayashi, for encouragements and helpful comments.

\end{document}